\documentclass[11pt]{article}
\usepackage{amsmath,amssymb,amsthm,graphicx,subfigure,float,url}
 \usepackage{paralist}
 \usepackage{graphics}
\usepackage{pdfsync}
\usepackage[colorlinks=true]{hyperref}

\renewcommand{\geq}{\geqslant}
\renewcommand{\leq}{\leqslant}

\newtheorem{theorem}{Theorem}[section]

\newtheorem{proposition}{Proposition}

\theoremstyle{definition}
\newtheorem{definition}[theorem]{Definition}
\newtheorem{remark}{Remark}

\date{}

\title{Some inverse problems around the tokamak \textsl{Tore Supra}\footnote{The second and the third authors are partially supported by the ANR project GAOS ``Geometric analysis of optimal shapes''}}
\author{Yannick Fischer\footnote{INRIA Sophia Antipolis M\'editerran\'ee, 2004 route des Lucioles, BP 93, 06902 Sophia-Antipolis, France;
yannick.fischer@sophia.inria.fr}\and 
Benjamin Marteau\footnote{ENSIMAG, 681, rue de la passerelle - Domaine universitaire - BP 72 - 38402 Saint Martin D'H\`eres, France;
marteaub@ensimag.fr}\and 
Yannick Privat\footnote{ENS Cachan Bretagne, CNRS, Univ. Rennes 1, IRMAR,
av. Robert Schuman, F-35170 Bruz, France; 
yannick.privat@bretagne.ens-cachan.fr}
}

\begin{document}

\maketitle
\begin{abstract}
We consider two inverse problems related to the tokamak \textsl{Tore Supra} through the study of the magnetostatic equation for the poloidal flux. The first one deals with the Cauchy issue of recovering in a two dimensional annular domain boundary magnetic values on the inner boundary, namely the limiter, from available overdetermined data on the outer boundary. Using tools from complex analysis and properties of genereralized Hardy spaces, we establish stability and existence properties. Secondly the inverse problem of recovering the shape of the plasma is addressed thank tools of shape optimization. Again results about existence and optimality are provided. They give rise to a fast algorithm of identification which is applied to several numerical simulations computing good results either for the classical harmonic case or for the data coming from \textsl{Tore Supra}.
\end{abstract}
\begin{center}
{\it This paper is dedicated to Michel Pierre on the occasion of his 60\textsuperscript{th} birthday.}
\end{center}
\vspace*{0.3cm}



\section{Introduction}

A very challenging potential application of inverse boundary problems for elliptic equations in doubly connected domains is related to plasma confinement for thermonuclear fusion in a tokamak \cite{Blum}.

First of all the tokamak is a concept invented in the Soviet Union and introduced in the 1950s. The meaning of the word is, after translation from the Russian words, ``toroidal chamber with magnetic coils'' \cite{Wesson} and simply refers to a magnetic confinement device with toroidal geometry (see Figure \ref{FigYF1}). Today, tokamaks are the most used devices for the fusion experiments and needless to say represent the most suitable approach for their control. In the present paper we focus on the case of the tokamak \textsl{Tore Supra} built at the CEA/IRFM Cadarache (France).

\medskip
In order to control the plasma position by applying feedback control, it is necessary to know its position in a very short time, or in other words in a time which has to be smaller than the sampling frequency of the plasma shape controller (some microseconds). As a consequence, the poloidal flux function has to be fast identified and computed in a very effective way since its knowledge in the domain included between the exterior wall of the tokamak and the boundary of the plasma is sufficient to determine completely the plasma boundary. Insofar as the only data available are the ones obtained by measurements uniformly distributed on the entire exterior wall of the tokamak, namely the poloidal flux function and its normal derivative, the identification of the plasma boundary may be viewed as the solution of a free boundary problem.

The computation of the poloidal flux function in the vacuum and the recovery of the plasma shape have already been extensively studied (see for example \cite{Blum}). Consequently, several numerical codes have been developed and we mention here the code Apollo actually used for the tokamak \textsl{Tore Supra} \cite{Saint} based on an expansion of Taylor and Fourier types for the flux. Others expansions may be found in the literature such as the one making use of toroidal harmonics involving Legendre functions \cite{ALACRI}. Naturally, those series expansions are truncated for computations and the coefficients are determined so that they fit to the measurements.

Needless to say, the inverse problem consisting in identifying Dirichlet as well as Neumann data inside a domain from such overdetermined data on part of the external boundary has already generated a lot of methods. We refer for instance to Tikhonov regularization \cite{Tikhonov}, iterativ method \cite{Kozlov}, conformal mappings \cite{Haddar}, integral equations \cite{Rundell}, quasi-reversibility \cite{bourgeois,Klibanov,Lattes} or level-set \cite{Ameur,burger,burger2}.

\medskip
The purpose of this paper is to formulate alternative and original methods for the resolution of the free boundary problem. Two different approaches are considered: a first one making use of Complex Analysis tools and a second one based on Shape Optimization. Up to the knowledge of the authors, such approaches have never been carried out before in the case of the tokamak (although many methods have been developed to solve some close inverse problems, see for instance \cite{Ameur,bourgeois,burger, Haddar,Rundell}).

Firstly, we are interested in the following inverse problem : assume that the domain under study, denoted by $\Omega_l$, is the vacuum located between the outer boundary $\Gamma_e$ of the tokamak and the limiter $\Gamma_l$ (see Section \ref{geom_tok}) in such a way that it refers to a fixed annular domain, or more precisely, to a conformally equivalent doubly-connected domain. The poloidal flux and its normal derivative being given on $\Gamma_e$, we want to recover those magnetic data on $\Gamma_l$ inside the device. This amounts to solve a Cauchy problem from overdetermined data on part of the boundary. This problem is known to be ill-posed since the work of Hadamard. However sufficient conditions on the available data on $\Gamma_e$, together with \textit{a priori} hypotheses on the missing data, may be provided for continuity and stability properties to hold. In order to deal with these constraints, a link is established between the equation satisfied by the poloidal flux in the vacuum and the conjugate Beltrami equation \cite{BLRR}. As a result of this part, we will be able to provide the flux and its normal derivative everywhere in the domain $\Omega_l$ then more particularly on the limiter $\Gamma_l$ itself.

In a second time, we investigate the inverse problem consisting in recovering the shape of the plasma inside the tokamak, from the knowledge of Cauchy data on the external boundary $\Gamma_e$ (or on $\Gamma_l$ if the resolution of the inverse problem of last paragraph has been performed first). Indeed, if the plasma domain were known, the poloidal flux $u$ inside the tokamak would be determined until the boundary of the plasma $\Gamma_p$ by an elliptic partial differential equation, namely the homogeneous magnetostatic equation for the poloidal flux, with Cauchy data on $\Gamma_e$. The question of the global existence of a solution for such a system is open and probably quite difficult. The point of view developed in this paper does not need such a result. Indeed, given that the shape of the plasma is a level set of the poloidal flux $u$, the question of the determination of the shape of the plasma inside the tokamak comes down intrinsically to solve an overdetermined partial differential equation. Noticing that, in general, an overdetermined condition can be interpreted as the first necessary optimality conditions of a Shape Optimization problem (see for instance \cite{allaire,HP}), we chose to see the shape of the plasma as a minimizer of a shape functional in a given class of admissible domains. Furthermore, we use a standard shape gradient algorithm to compute numerically the boundary $\Gamma_p$.

We point out that the Shape Optimization part has led to numerical developpements that have been used in the present paper. However an effective numerical method based on the Complex Analysis part is under progress. This latter would extend some recent results obtained for the Laplace operator \cite{Haddar} to the case of more general elliptic operators and could probably help to initialize the part treating with the shape of the plasma (see discussion at end of Section \ref{sec:BEP}).

\medskip
The overview of the article is as follows. Section \ref{sec:plasequi} is devoted to general notations and a description of a model of plasma equilibrium. The equations governing this latter allow to characterize the boundary $\Gamma_p$ of the plasma as a level line of the flux in the domain bounded by the limiter. A brief description of the geometry of \textsl{Tore Supra} is provided too. Our main existence and stability results for the Cauchy problem explained above are stated in Section \ref{sec:magdata}. We introduce some generalized harmonic functions associated with the problem, which in fact belong to generalized Hardy spaces of an annulus. A useful density result in such classes is given which enables the resolution of the Cauchy problem formulated as a bounded extremal one. Finally, in Section \ref{sec:shape}, we introduce a Shape Optimization problem whose solution is supposed to be the shape of the plasma. We first prove the existence of minimizers for such a problem and state the associated necessary first order optimality conditions. In particular, it is shown that, under a regularity assumption for the optimum, the solution of this problem verifies the system characterizing the shape of the plasma. Some numerical simulations, based on an optimization algorithm and the use of the shape derivative, are included at the end of the paper and permit to recover in a satisfying way the boundary $\Gamma_p$ of the domain occupied by the plasma.

\section{A model of plasma equilibrium} \label{sec:plasequi}

\subsection{The magnetostatic equation for the poloidal flux}\label{grad_shaf}
We denote by $(r,\varphi,z)$ the three-dimensional cylindrical coordinates system where $r$ is the radial coordinate, $\varphi$ is the toroidal angle and $z$ is the height; $\vec{e_r},\vec{e_\varphi}$ and $\vec{e_z}$ will denote the axis unit vectors. Thus, given a generic vector $\vec{A}$, its component along the unit vectors will be denoted by $A_r, A_\varphi$ and $A_z$, respectively, so as to have

\begin{equation*}
\vec{A} = A_r \vec{e_r} + A_\varphi \vec{e_\varphi} + A_z\vec{e_z}.
\end{equation*}

Since the tokamak is an axisymmetric toroidal device, we may assume that all magnetic quantities do not depend on the toroidal angle $\varphi$. That involves that a plasma equilibrium may be studied in any cross section $(r,z)$, named poloidal section (see Figure \ref{FigYF1}).
\begin{figure}
\begin{center}
\includegraphics{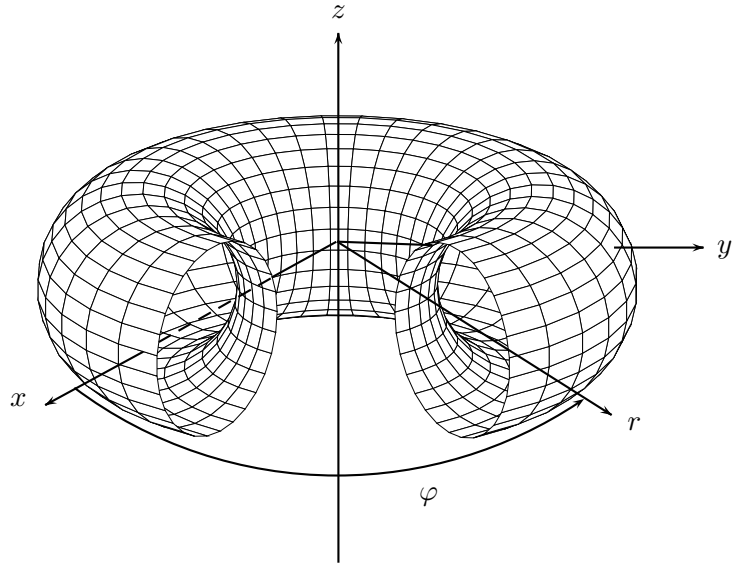}
\vspace{-1cm}\caption{Toroidal geometry of the tokamak Tore Supra.}
\label{FigYF1}
\end{center}
\end{figure}
Now writing $\vec{B}$ for the magnetic induction and $u$ for the poloidal magnetic flux, and denoting by $\nabla \cdot \vec{B} = \tfrac{\partial B_r}{\partial r}+\tfrac{\partial B_z}{\partial z}$, $\nabla \times \vec{B} = \tfrac{\partial B_z}{\partial r}-\tfrac{\partial B_r}{\partial z}$ and $\nabla u = (\tfrac{\partial u}{\partial r},\tfrac{\partial u}{\partial z})$, it is easy to derive from Maxwell's equations \cite{Blum}, especially from Gauss's law $\nabla \cdot \vec{B} =0$, the following equations

\begin{equation} \label{YF:eq1}
B_r = -\dfrac{1}{r} \dfrac{\partial u}{\partial z} \quad \text{and} \quad B_ z = \dfrac{1}{r} \dfrac{\partial u}{\partial r}.
\end{equation}

Combining those relations with Ampere's law in the vacuum region located between the outer boundary of the tokamak and the limiter (see Section \ref{geom_tok}), i.e

\begin{equation} \label{YF:eq2}
\nabla \times \vec{B} = 0,
\end{equation}
we obtain the following equation for the poloidal flux $u$ in the vacuum

\begin{equation} \label{YF:eqplasma}
\nabla \cdot (\dfrac{1}{r} \nabla u) = 0.
\end{equation}

It is a linear homogeneous second order elliptic equation in divergence form in two dimensions. Observe that this latter does not stand in the whole space located inside the tokamak owing to the presence of a current density $\vec{j}$ in the plasma. Here, denoting by $j_T$ its toroidal component, equation (\ref{YF:eqplasma}) becomes the so-called Grad-Shafranov \cite{GRAD} equation
\begin{equation} \label{YF:eq3}
\nabla \cdot (\dfrac{1}{r} \nabla u) = j_T,
\end{equation}
which is a non-linear elliptic equation in two dimensions. However observe that in the framework if this paper, we only consider equation (\ref{YF:eqplasma}).

\subsection{The geometry of \textsl{Tore Supra}} \label{geom_tok}

We present here the main features of the tokamak \textsl{Tore Supra} built at CEA/IRFM Cadarache in France in a poloidal cross section $(r,z)$. All numeric quantities for the distances are given in meters.

\begin{figure}
\begin{center}
\includegraphics{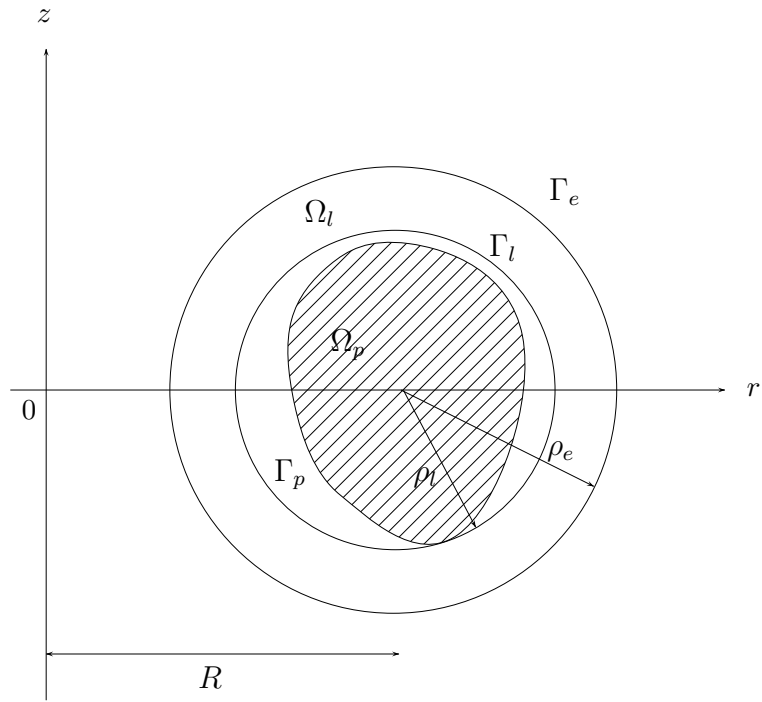}
\caption{Schematic representation of a poloidal section of Tore Supra.}
\label{FigYF2}
\end{center}
\end{figure}

Here $\Gamma_e$ stands for the outer boundary of the tokamak. It corresponds to a circle with center $(R,0) = (2.42,0)$ and radius $\rho_e = 0.92$. Inside the tokamak, the first material piece encountered is the limiter $\Gamma_l$. It corresponds to a closed curve and its goal is to prevent any interaction between the plasma and $\Gamma_e$. In a nutshell, the plasma cannot pass through $\Gamma_l$. Inside the limiter stands the plasma, referring to the domain $\Omega_p$ with a boundary $\partial \Omega_p = \Gamma_p$.

Furthermore we add two notations. In the following we will denote by $\Omega_v$ the domain occupied by the vacuum, i.e the one included between $\Gamma_p$ and $\Gamma_e$. Moreover $\Omega_l$ will denote the domain included between $\Gamma_l$ and $\Gamma_e$. Then refering to (\ref{YF:eqplasma}), the equation satisfied by the poloidal magnetic flux $u$ in $\Omega_v$, whence {\it{a fortiori}} in $\Omega_l$,  is $ \nabla \cdot (\tfrac{1}{r} \nabla u) = 0$.

It should be noticed that in \textsl{Tore Supra}, the domain $\Omega_l$ is not exactly an annulus as it can be observed on Figure \ref{fig4}. But as explained in \cite[Section 6]{BLRR}, all properties of functions solutions to (\ref{YF:divu}) below are preserved by composition with conformal maps. Accordingly, the results of the next section will always refer to a annular domain like in Figure \ref{FigYF2}.

\subsection{Definition of the plasma boundary \texorpdfstring{$\Gamma_p$}{Gammap}} \label{sec:defplasmabound}

A commonly accepted definition of the plasma boundary $\Gamma_p$ may be given in these terms \cite{Blum}: it is the outermost closed magnetic surface entirely contained in the limiter $\Gamma_l$. Moreover it corresponds to an equipotential of the function $u$. In fact, two situations can happen for the plasma equilibrium in tokamak devices:

\medskip
\begin{itemize}
\item the limiter case. Here $\Gamma_p$ and $\Gamma_l$ have at least one common point. The equipotential's value of $u$ defining the plasma boundary $\Gamma_p$ is given by the maximum of $u$ on the limiter $\Gamma_l$. At all contact points, the limiter and the plasma are tangent.
\item the divertor case. $\Gamma_p$ and $\Gamma_l$ have no point in common. In this case, the plasma is strictly included inside the limiter and its boundary has an X-point (hyperbolic point).
\end{itemize}

Let us specify that we made no attempt to study the divertor case. In fact, the paper focuses only on the limiter case, more particularly when the plasma and the limiter have only one point in common. However it should be noticed that this study could easily be applied in the case of several common points. We shall make use of this observation without further notice.

\section{How to recover magnetic data in the tokamak \textsl{Tore Supra} ?} \label{sec:magdata}

Let $\Omega_l$ be the bounded open domain in $\mathbb{R}^2$ with $C^1$ boundary mentioned in Section \ref{geom_tok} and
\begin{equation}
\label{YF:condsigma}
\sigma(r,z) = \dfrac{1}{r} \in W^{1,\infty}(\Omega_l;\mathbb{R}) \quad \text{such that} \quad 0 < C_1 \leq \sigma(r,z) \leq C_2 \quad \text{a.e in} \quad \Omega_l
\end{equation}
for two constants $ 0 < C_1 < C_2 < +\infty $. The elliptic equation we look at is (\ref{YF:eqplasma}) in the vacuum, i.e
\begin{equation}
\label{YF:divu}
\nabla \cdot (\dfrac{1}{r} \nabla u) = 0 \quad \text{a.e. in} \quad \Omega_l.
\end{equation}
Here the differential operator $\nabla \cdot$ and $\nabla$ are understood in the sense of distributions (with respect to the real variables $r,z$ in $\mathbb{R}^2$). In this section we are dealing with the Cauchy problem described in the introduction. This problem is, as defined by Hadamard, ill-posed for magnetic data $u$ and $\partial_n u$ prescribed on $\Gamma_e$ not only in Sobolev spaces $W^{1-1/p,p}(\Gamma_e), 1 < p < +\infty$ but also in $L^p(\Gamma_e)$, where they still make sense \cite{BLRR}. However, well-posedness may be ensured if $L^2$-norm constraints is added on $\Gamma_l$. It follows that the extrapolation issue from boundary data turned out as a best approximation one on $\Gamma_e$, still named bounded extremal problem (BEP).

Note that from now on we restrict ourselves to the particular Hilbertian framework, i.e for $p=2$, in order to deal with any possible boundary data on $\Gamma_e$ having finite energy.

\begin{remark}
When $\sigma$ is constant, this problem reduces to solving a Cauchy problem for the Laplace operator, that is to say to recover the values of a holomorphic function in a domain of analyticity from part of its boundary values. It is well worth noting that this problem has been extensively studied in Hardy spaces on simply and doubly connected domains (see \cite{BL1,BL2,JLMP}).
\end{remark}

\subsection{From Grad-Shafranov equation to conjugate Beltrami equation}

To study the Cauchy problem for (\ref{YF:divu}), our approach proceeds via a complex elliptic equation of the first order, namely the conjugate Beltrami equation:
\begin{equation}
\label{YF:CB}
\overline{\partial} f = \nu \overline{\partial f} \quad \text{a.e in} \quad \Omega_l,
\end{equation}
where $\nu \in W^{1,\infty}(\Omega_l;\mathbb{R})$ satisfies
\begin{equation}
\label{YF:condnu}
\|\nu\|_{L^{\infty}(\Omega)} \leq \kappa \quad \text{for some} \quad \kappa \in (0,1)
\end{equation}
and where we use the standard notations
\begin{equation*}
\partial f = \dfrac{1}{2} (\dfrac{\partial}{\partial r} - i \dfrac{\partial}{\partial z}) f \quad \text{and} \quad \overline{\partial} f = \dfrac{1}{2} (\dfrac{\partial}{\partial r} + i \dfrac{\partial}{\partial z}) f.
\end{equation*}

From this point, it is easy to verify that (\ref{YF:CB}) decomposes into a system of two real elliptic equations of the second order in divergence form. Indeed, assume first that $f=u+iv$ is a solution to (\ref{YF:CB}) with real valued $u$ and $v$. Putting this in (\ref{YF:CB}) yields that $u$ satisfies (\ref{YF:divu}) while $v$ satisfies
\begin{equation}
\label{YF:divv}
\nabla \cdot (\dfrac{1}{\sigma} \nabla v) = 0 \quad \text{a.e. in} \quad \Omega_l
\end{equation}
with $\sigma = \tfrac{1-\nu}{1+\nu}$. Moreover, from that definition of $\sigma$, we obtain that (\ref{YF:condnu}) implies (\ref{YF:condsigma}). Conversely, assume that $u$ is real valued and satisfies (\ref{YF:divu}). By putting $\nu = \tfrac{1-\sigma}{1+\sigma}$, a simple computation shows that (\ref{YF:CB}) is equivalent to the system
\begin{equation}
\label{YF:CR}
\left\{ \begin{array}{lll}
\dfrac{\partial v}{\partial r} &=& -\sigma \dfrac{\partial u}{\partial z} \\
& & \quad \quad \quad \quad \quad \quad \text{a.e. in} \quad \Omega_l. \\
\dfrac{\partial v}{\partial z} &=& \sigma \dfrac{\partial u}{\partial r}
\end{array} \right.
\end{equation}
This system admits a solution when $\Omega_l$ is simply connected. Indeed the differential form defined by $W = (W_1,W_2)=(\sigma \tfrac{\partial u}{\partial r},-\sigma \tfrac{\partial u}{\partial z})$ satisfies $\tfrac{\partial W_1}{\partial r} = \tfrac{\partial W_2}{\partial z}$. Consequently it is closed and applying Poincar\'e's Lemma yields the existence of a real valued function $v$, unique up to an additive constant, such that $f=u+iv$ satisfies (\ref{YF:CB}) with the definition of $\nu$ given above. Again, conditions (\ref{YF:condsigma}) implies (\ref{YF:condnu}).

\medskip
Note that when $\Omega_l$ is doubly connected, such like in the situation of Tore Supra, the existence of $v$ as a real single-valued function may only be local. Indeed, from Green's formula applied to $u$ and to any constant function in $\Omega_l$, we get
\begin{equation*}  
0 = \int_{\Omega_l} \nabla \cdot (\sigma \nabla u) = \int_{\partial \Omega_l} \sigma \dfrac{\partial u}{\partial n} = \int_{\Gamma_e}\sigma \dfrac{\partial u}{\partial n} + \int_{\Gamma_l}\sigma \dfrac{\partial u}{\partial n}.
\end{equation*}
But, since (\ref{YF:CR}) may be rewritten on $\partial \Omega_l$ as $\tfrac{\partial v}{\partial t} = \sigma \tfrac{\partial u}{\partial n}$ where $\tfrac{\partial}{\partial t}$ stands for the tangential partial derivative on $\partial \Omega_l$, it holds that 
\begin{equation*}  
-\int_{\Gamma_e}\sigma \dfrac{\partial u}{\partial n} = \int_{\Gamma_l} \dfrac{\partial v}{\partial t}.
\end{equation*}
Hence $v$ is clearly multiplied-valued if $\int_{\Gamma_e}\sigma \tfrac{\partial u}{\partial n} \neq 0$ and $f=u+iv$ \textit{a fortiori} too. However it is always possible to define $u$ as the real part of a single valued function $f$ in $\Omega_l$. See for example \cite{JLMP} where the holomorphic case is processed by the aid of the logarithmic function (well-known to be multiplied-valued in the unit disk). We only mention, without giving more details, that our situation may be similarly solved by making use of toroidal harmonics \cite{ARPI}.

\subsection{Generalized Hardy spaces}\label{YF:genHardySpaces}
Assume that $\nu \in W^{1,\infty}(\Omega_l;\mathbb{R})$ verifies (\ref{YF:condnu}). When handling boundary data in the fractional Sobolev space $W^{1/2,2}(\partial \Omega_l)$, it is a result from \cite{Cam} that the Dirichlet problem for (\ref{YF:CB}) admits a unique solution in $W^{1,2}(\Omega_l)$. But as mentioned before, the assumptions on the boundary data are relaxed in a manner that they now belong to the Lebesgue space $L^2(\partial \Omega_l)$. In this case, it is obvious that the solution of the Dirichlet problem for (\ref{YF:CB}) has no more reason to belong to $W^{1,2}(\Omega_l)$ in general. On the other hand the Dirichlet problem will rather have a solution in some generalized Hardy spaces whose definition will follow.

This focuses the attention on the fact that those generalized Hardy spaces are the natural spaces when trying to solve inverse problem linked to (\ref{YF:CB}) for $L^2$ boundary data (see \cite{FLPS} for the simply connected case). Moreover, in Section \ref{sec:BEP}, we will show that the Cauchy problem may be reformulated as a bounded extremal one admitting a unique solution in those appropriate Hardy classes.

\medskip
Let us denote by $\mathbb{D}_{R,\rho}$ and $\mathbb{T}_{R,\rho}$ the disk and the circle centered at $(R,0)$ of radius $\rho$. Remember that from now $\Omega_l$ stands for the annular domain shown in Figure \ref{FigYF2}. Thus we have $\mathbb{T}_{R,\rho_e} = \Gamma_e$.

\begin{definition}
If $\nu \in W^{1,\infty}(\Omega_l;\mathbb{R})$, we denote by $H^2_\nu(\Omega_l)$ the generalized Hardy space which consists in Lebesgue measurable functions $f$ on $\Omega_l$, solving (\ref{YF:CB}) in the sense of distributions in $\Omega_l$ and satisfying
\begin{equation}
\label{YF:normHnu}
\|f\|_{H^2_\nu(\Omega_l)} := \underset{\rho_l < \rho < \rho_e}{ess \sup} \ \|f\|_{L^2(\mathbb{T}_{R,\rho})}  < +\infty
\end{equation}
where
\begin{equation*}
\|f\|_{L^2(\mathbb{T}_{R,\rho})} := \left(\dfrac{1}{2\pi}\int_0^{2\pi} |f(R+\rho e^{i\theta})|^2 \ d\theta \right)^{1/2}
\end{equation*}
\end{definition}
Equipped with the norm defined by (\ref{YF:normHnu}), $H^2_\nu(\Omega_l)$ is a Hilbert space. Moreover (\ref{YF:normHnu}) implies $f \in L^2(\Omega_l)$.

We remark immediately that when $\nu =0$ and $\Omega = \mathbb{D} = \mathbb{D}_{0,1}$, $H^2_\nu(\Omega_l)$ is nothing but the classical $H^2(\Omega_l)$ space of holomorphic functions on the unit disk bounded in norm $L^2$ on $ \mathbb{T} = \mathbb{T}_{0,1}$ (see \cite{Duren,Gar}). Recall just that it consists in the functions $f \in L^2(\mathbb{T})$ which Fourier coefficients of negative order vanish. Most of the properties of generalized Hardy spaces on simply connected domains derive from those of the classical $H^2$ spaces and still hold for multiply connected domains. Basically, the main idea relies on the connection between functions $f \in H^2_\nu(\Omega_l)$ and functions $\omega$ satisfying in the distributional sense, for $\alpha \in L^{\infty}(\Omega_l)$,
\begin{equation}
\label{YF:omega}
\overline{\partial} \omega = \alpha \overline{\omega} \quad \text{on} \quad \Omega_l
\end{equation}
and such that the condition (\ref{YF:normHnu}) still holds. That connection has been introduced in \cite{BN} and leads to
\begin{proposition}
Let $\nu \in W^{1,\infty}(\Omega_l;\mathbb{R})$ satisfy (\ref{YF:condnu}) and define $\alpha \in L^{\infty}(\Omega_l)$ by
\begin{equation*}
\alpha = -\dfrac{\overline{\partial} \nu}{1-\nu^2} = \overline{\partial} \log \left[\dfrac{1-\nu}{1+\nu}\right]^{1/2} =  \overline{\partial} \log \sigma^{1/2}
\end{equation*}
then $f=u+iv \in H^2_\nu(\Omega_l)$ if and only if $\omega = \dfrac{f - \nu \overline{f}}{\sqrt{1-\nu^2}} = \sigma^{1/2}u + i\sigma^{-1/2}v$ solves (\ref{YF:omega}) with condition (\ref{YF:normHnu}).
\end{proposition}

\begin{proof}
The proof is a straightforward computation. Indeed assume $f \in H^2_\nu(\Omega_l)$. Then from $\omega = \tfrac{f - \nu \overline{f}}{\sqrt{1-\nu^2}}$ we get
\begin{equation*}
\overline{\partial} \omega = \dfrac{\nu \overline{\partial} \nu}{(1-\nu^2)^{3/2}} (f-\nu \overline{f}) + \dfrac{1}{(1-\nu^2)^{1/2}} (\overline{\partial} f - \nu \overline{\partial f} - \overline{f} \ \overline{\partial} \nu).
\end{equation*}
Insofar as $f \in H^2_\nu(\Omega_l)$, it is true that $\overline{\partial} f - \nu \overline{\partial f} =0$. Hence
\begin{equation*}
\begin{split}
\overline{\partial} \omega =& \dfrac{\nu \overline{\partial} \nu}{(1-\nu^2)^{3/2}} (f-\nu \overline{f}) - \dfrac{1}{(1-\nu^2)^{1/2}} \overline{f} \  \overline{\partial} \nu \\
=& - \dfrac{\overline{\partial} \nu}{(1-\nu^2)^{3/2}}(\overline{f} - \nu f)  = - \dfrac{\overline{\partial} \nu}{1-\nu^2} \omega,
\end{split}
\end{equation*}
in other words $\overline{\partial} \omega = \alpha \overline{\omega}$. From the definition of $\omega$, it is obvious that (\ref{YF:normHnu}) still holds for this latter. By the same way, and noticing that $f = \dfrac{\omega + \nu  \ \overline{\omega}}{\sqrt{1-\nu^2}}$, the converse is valid.
\end{proof}

Furthermore solutions of type $\omega$, which satisfy condition (\ref{YF:normHnu}), can be represented as $\omega = e^s g$ where $s$ is continuous on $\overline{\Omega_l}$ and $g \in H^2(\Omega_l)$ \cite{BLRR,BN}. This remark allows to establish that on simply connected domains, classical and generalized Hardy spaces share similar properties and those latter are naturally exportable to multiply connected domains. We recall in the following result most of them

\begin{proposition}
\label{YF:prophardy}
If $\nu \in W^{1,\infty}(\Omega_l;\mathbb{R})$ satisfies (\ref{YF:condnu}),
\begin{enumerate}
\item any function $f \in H^2_\nu(\Omega_l)$ has a non-tangential limit almost everywhere on $\partial \Omega_l = \Gamma_e \cup \Gamma_l$, called the trace of $f$ and denoted by $tr f$, which belongs to $L^2(\partial \Omega_l)$;

\medskip
\item $\|tr f\|_{L^2(\partial \Omega_l)}$ defines an equivalent norm on $H^2_\nu(\Omega_l)$;

\medskip
\item $tr H^2_\nu(\Omega_l)$ is closed in $L^2(\partial \Omega_l)$;

\medskip
\item if $f \in H^2_\nu(\Omega_l)$, $tr f$ cannot vanish on a subset of $\partial \Omega_l$ with positive measure unless $f \equiv 0$;

\medskip
\item each $f \in H^2_\nu(\Omega_l)$ satisfies the maximum principle, i.e $|f|$ cannot assume a relative maximum in $\Omega_l$ unless it is constant.
\end{enumerate}

\end{proposition}
We refer to \cite{BLRR} for the proofs in simply connected domains.

\medskip
These properties are necessary to prove the results of next section, i.e. the density of traces of function in $H^2_\nu(\Omega_l)$ in $L^2(I)$ whenever $I$ is a subset of non-full measure of the boundary $\partial \Omega_l = \Gamma_e \cup \Gamma_l$, which are the key point to solve extremal problems with incomplete boundary data.

\subsection{Density results and bounded extremal problem} \label{sec:BEP}

This section is devoted to the resolution of the Cauchy problem formulated in the introduction.The first step which needs to be established is a density result of fundamental importance which asserts that if $I \subset \partial \Omega_l$ has positive measure as well as $J = \partial \Omega_l \setminus I$, then every $L^2$-complex function on $I$ can be approximated by the trace of a function in $H^2_\nu(\Omega_l)$. Remember that in the situation under study in this paper, the subset $I$ of the boundary $\partial \Omega_l$ corresponds \textit{a priori} to the whole boundary $\Gamma_e$. But in order to ensure density results as well as better algorithm's convergence (see the end of Section \ref{sec:BEP}), we restrict from now on the set of available magnetic data to a strict subset, still named  $I$ of the outer boundary $\Gamma_e$. We mention that the case where $I$ is effectively the whole boundary $\Gamma_e$ is already under study. 

Thus the following result is a natural extension of \cite[Theorem 4.5.2.1]{BLRR} for an annular domain.

\begin{theorem}
\label{YF:thdense}
Let $I \subset \partial \Omega_l = \Gamma_e \cup \Gamma_l$ be a measurable subset such that $J = \partial \Omega_l \setminus I$ has positive Lebesgue measure. The restrictions to $I$ of traces of $H^2_\nu$-functions are dense in $L^2(I)$.
\end{theorem}
From the magnetic data on $I$, one can compose the function $F_d = u + i v$ where $v = \int_{\Gamma_e} \sigma \tfrac{\partial u}{\partial n}$ according to the formulation of (\ref{YF:CR}) on $\partial \Omega_l$. As a consequence of Theorem \ref{YF:thdense}, it exists a sequence of functions $(f_n)_{n \geq 0} \in H^2_\nu(\Omega_l)$ which trace on $I$ converges to $F_d$ in $L^2(I)$. Suppose now that $\underset{n \to +\infty}{\lim} \|tr f_n\|_{L^2(J)} \neq +\infty$. Up to extracting a subsequence, we may assume that $(tr f_n)_{n \geq 0}$ is bounded in $L^2(J)$. But $(tr f_n)_{n \geq 0}$ is already bounded in $L^2(I)$ by assumption. Consequently it remains bounded in $L^2(\partial \Omega_l)$ and then in $H^2_\nu(\Omega_l)$ by point $2$ of Proposition \ref{YF:prophardy}. As $H^2_\nu(\Omega_l)$ is a Hilbert space, there exists a subsequence $(tr f_{n_p})_{p \geq 0}$ converging weakly toward some $f \in H^2_\nu(\Omega_l)$. Now, by restriction, it is obvious that $(tr {f_{n_p}}_{|I})_{p \geq 0}$ converges weakly to $f_{|I}$ in $L^2(I)$. Since $F_d$ is the strong limit of $(tr f_n)_{n \geq 0}$ in $L^2(I)$, we conclude that $F_d = f$ a.e on $I$. For this reason $F_d$ can either already be the trace on $I$ of a $H^2_\nu$-function, or $\|tr f_n\|_{L^2(J)} \to +\infty$ as $n \to +\infty$. This behaviour of the approximant on $J$ shows that the Cauchy solution is unstable, which in fact explains that the inverse problem is ill-posed. Note that in general the function $F_d$ does not coincide with the trace of a function belonging to $H^2_\nu(\Omega_l)$ insofar as magnetic data proceeds from sensor subject, like for all mechanical engineering, to roundoff errors.

Therefore, in order to prevent such an unstable behaviour, an upper bound for the $L^2$ norm of the approximation on $J$ will be added. Thus the Cauchy problem may be expressed as a bounded extremal one in appropriate Hardy classes. In practical terms, define, for $M > 0$ and $\phi \in L^2(J)$,
\begin{equation*}
\mathcal{B}_{\Omega_l} = \left\{ g \in tr H^2_\nu(\Omega_l); \|Re \, g - \phi\|_{L^2(J)} \leq M \right\}_{|I} \subset L^2(I).
\end{equation*}
In light of this definition, it is now possible to find a unique solution for the minimization problem formulated on the class $\mathcal{B}_{\Omega_l}$.
\begin{theorem}
\label{YF:thBEP}
Fix $M > 0, \phi \in L^2(J)$. For every function $F_d \in L^2(I)$, there exists a unique solution $g^* \in \mathcal{B}_{\Omega_l}$ such that
\begin{equation*}
\|F_d - g^*\|_{L^2(I)} = \underset{g \in \mathcal{B}_{\Omega_l}}{\min} \|F_d - g\|_{L^2(I)}.\\
\end{equation*}
Moreover, if $F_d \notin \mathcal{B}_{\Omega_l}$, the constraint is saturated, i.e $\|Re \, g^* - \phi\|_{L^2(J)} = M$.
\end{theorem}

\begin{proof}
Since $\mathcal{B}_{\Omega_l}$ is clearly convex and the norm $\|.\|_{L^2(\partial \Omega_l)}$ lower semi-continuous, it is enough to show that $\mathcal{B}_{\Omega_l}$ is closed in $L^2(I)$ to ensure the existence and the uniqueness of $g^*$. Indeed, there is a best approximation on any closed convex subset of a uniformly convex Banach space which is in this case $L^2(I)$ \cite[Part 3, Chapter II, 1, Propositions 5 and 8]{Beau}.

Let $(\varphi_{k|_{I}})_{k \geq 0} \in \mathcal{B}_{\Omega_l}$, $\varphi_{k|_{I}} \to \varphi_{I}$ in $L^2(I)$ as $k \to \infty$. 
Put $u_k = Re \, \varphi_k$. By assumption, $(u_k)_{k \geq 0}$ is bounded in $L^2(\partial \Omega_l)$. Then the application of norm's inequalities contained in \cite[Theorem 4.4.2.1]{BLRR} shows that $\|\varphi_k\|_{H^2_\nu(\Omega_l)} \leq c_\nu \|u_k\|_{L^2(\partial \Omega_l)}$. Hence $(\varphi_k)_{k \geq 0}$ is bounded in $L^2(\partial \Omega_l)$ and, up to extracting a subsequence, weakly converges to $\psi \in L^2(\Omega_l)$; 
necessarily 
$\psi_{|_{I}} = \varphi_{I}$. 

Next, $\varphi_k \in tr H^{2}_\nu(\Omega_l)$, which is weakly closed in $L^2(\partial \Omega_l)$ (see discussion after Theorem \ref{YF:thdense}); this implies that $\psi \in tr H^{2}_\nu(\Omega_l)$. Because $\mbox{Re } \varphi_k = u_k$ satisfies the constraint on $J$, so does $Re \, \psi$, whence $\varphi_{I}  \in \mathcal{B}_{\Omega_l}$. 

Let us now prove that, if $F_d\notin \mathcal{B}_{\Omega_l}$, then 
$\left\Vert Re \, g^* - \phi\right\Vert_{L^{2}(J)}=M$. 
Assume for a 
contradiction that $\left\Vert Re \, g^* - \phi\right\Vert_{L^{2}(J)}<M$. Since $\|F_d-g^*\|_{L^2(I)} > 0$, by Theorem \ref{YF:thdense}, there is a function $h\in tr H^2_\nu(\Omega_l)$ such that 
\begin{equation*}
\|F_d-g^*- h\|_{L^2(I)}<\|F_d-g^*\|_{L^2(I)}.
\end{equation*}
By making use of the triangle inequality,
\begin{equation*}
\begin{split}
\|F_d-g^*-\lambda h\|_{L^2(I)} =& \|\lambda(F_d-g^*- h) + (1-\lambda)(F_d-g^*)\|_{L^2(I)} \\
\leq& |\lambda|\|F_d-g^*- h\|_{L^2(I)} + |1-\lambda|\|F_d-g^*\|_{L^2(I)} \\
<&  \|F_d-g^*\|_{L^2(I)}
\end{split}
\end{equation*}
for all $\lambda$ such that $0 < \lambda < 1$. Now, taking advantage of the assumption, we have that for $\lambda>0$ sufficiently small $\|Re \, (g^* + \lambda  h )- \phi \|_{L^2(J)} \le M$. To recap, $g^* + \lambda  h \in \mathcal{B}_{\Omega_l}$ and $ \|F_d-(g^*+\lambda h)\|_{L^2(I_1)} < \|F_d-g^*\|_{L^2(I_1)}$. This contradicts the optimality of $g^*$. 
\end{proof}
An explicit formulation of $g^*$ has been established in \cite{ABL} in the case $\nu=0$ which is still valid in the present situation. Indeed, denote by $\mathcal{P}_\nu$ the orthogonal projection from $L^2(\partial \Omega_l)$ onto $tr H^{2}_\nu(\Omega_l)$ (the operator $\mathcal{P}_\nu$ is the natural extension of the classical Riesz projection from $L^2(\mathbb{T})$ onto $tr H^{2} (\mathbb{D})$ \cite{Duren}). The unique solution of the bounded extremal problem, when $F_d \notin \mathcal{B}_{\Omega_l}$, is given by
\begin{equation} 
\mathcal{P}_\nu(\chi_{I_1} g^*) + \lambda \mathcal{P}_\nu(\chi_{J} Re \, g^*) = \mathcal{P}_\nu(\chi_{I_1} F_d) + \lambda \mathcal{P}_\nu(\chi_{J} \phi)
\end{equation}
where $\lambda > 0$ is the unique (Lagrange-type) parameter such that the constraint on $J$ is saturated, in other words $\|Re \, g^* - \phi\|_{L^2(J)} = M$ and $\chi_{J}$ is the characteristic function of $J$. Otherwise, the case where $F_d \in \mathcal{B}_{\Omega_l}$ corresponds to $\lambda = -1$. The behaviour with respect to $\lambda$ of the error $e(\lambda) = \|F_d-g^*\|^2_{L^2(I_1)}$  and the constraint $M$ is discussed in \cite{ABL}.

Observe that $g^*$ may be computed constructively if a complete family of $tr H^{2}_\nu(\Omega_l)$ is known. Indeed, the operator $\mathcal{P}_\nu$ can be expressed thanks to the conjugation operator which, for every function in $H^{2}_\nu(\Omega_l)$, associates the trace of $Re \, f$ to the trace of $Im \, f$. A formula for this latter may be found in \cite{AP}. Thus it is sufficient to expand any function in $L^2(\partial \Omega_l)$ on a basis for which the conjugation operator is easily computed. This method is described in \cite{FLPS} for simply connected domains.

\medskip
Finally we briefly describe a algorithm for solving numerically the bounded extremal problem. It has been developed in \cite{CBL} and consists in taking advantage of the fact that initially the magnetic data are available on the entire outer boundary $\Gamma_e$. As a first step compute the function $F_d$ thank the magnetic data on $\Gamma_e$ (as it is explained just after Theorem \ref{YF:thdense}) and split the measurement set into two disjoint parts $\Gamma_e = I_1 \cup I_2$. Secondly, solve the bounded extremal problem with respect to $\chi_{I_1}F_d$ and a constraint $M_1 > 0$. One obtain a solution $g^*_1$ which depends of $M_1$. From this latter compute the real number $M_2 = \underset{M_1 > 0}{\text{Argmin}} \ \|g^*_1 - F_d\|_{L^2(I_2)}$. And lastly solve the bounded extremal problem with respect to $\chi_{I_1}F_d$ and the constraint $M_2 > 0$. This cross validation procedure provides good results in the harmonic case, i.e. for the Laplace's equation. 

\medskip
We conclude this section by claiming that the solution $g^*$ of the bounded extremal problem could be a starting point of the next section which focuses on the recovery of the plasma boundary by technics of Shape Optimization. Indeed $g^*$ provides magnetic data, i.e. $u$ and $\partial_n u$, at every point of $\Omega_l$. It is then possible to initiate the free boundary problem with an outer boundary $\Gamma_E$ located in $\Omega_l$, which is then closer to the potential boundary $\Gamma_p$ of the plasma. Moreover this new boundary $\Gamma_E$ can still be $\Gamma_l$ itself.

An other feature, under the assumption that the point of contact $M_0$ between the plasma and the limiter is known in advance, is to provide the constant $c = g^*(M_0)$ corresponding to the value of the flux at this point.

\section{Recovering the shape of the plasma inside the tokamak \textsl{Tore Supra}} \label{sec:shape} 
\subsection{A Shape Optimization frame}
\par In this section we make use of the notations introduced on Figure \ref{FigYF2}. In particular, $\Gamma_e$ denotes the external boundary of the disk centered at $(r,z)=(R,0)$ with given radius $\rho_e >0$, representing the external wall of the tokamak \textsl{Tore Supra}. We are interested in the determination of the shape of the plasma inside the tokamak. In the following analysis, we will assume that the plasma is composed of one simply connected component, whose boundary is denoted $\Gamma_p$. Let us introduce $\Omega$, the domain located between the two closed curves $\Gamma_e$ and $\Gamma_p$. Physically, as it is underlined in Section \ref{sec:defplasmabound}, the only interesting case for our study is the limiter one. It means that there exists a point of the plasma boundary denoted by $M_0$, where the plasma meets the toroidal belt limiter $\Gamma_p$ and the poloidal magnetic flux is maximal. 
\par We have seen in Section \ref{sec:defplasmabound} that the boundary $\Gamma_p$ may be seen as a level set of the solution. In other words, we look for a domain $\Omega$ and its boundary $\Gamma_p=\partial\Omega\backslash\Gamma_e$ such that
\begin{equation}\label{eq0}
\left\{
\begin{array}{ll}
\nabla \cdot \left(\sigma \nabla u \right)=0, & x\in \Omega\\
u=u_0, & x\in \Gamma_e \\
\frac{\partial u}{\partial n}=u_1, & x\in \Gamma_e\\
\Gamma_p=\{u=u(M_0)\}.
\end{array}\right.
\end{equation}
where $u$ still denotes the poloidal magnetic flux inside the tokamak, $\nabla$ denotes the \textsl{nabla} operator with respect to the variables $(r,z)$, and $\sigma$ is a given analytic function such that $0<C_1<\sigma<C_2$ in $\Omega$. We recall that $M_0$ is such that $u(M_0)=\displaystyle \max_{x\in \overline{\Omega}}u(x)$ ($M_0$ is not unique {\it a priori}). Indeed, let us recall that in the case of the tokamak {\sl Tore Supra}, one has (see Section \ref{YF:genHardySpaces})
$$
\sigma(r,z)=\frac{1}{r}\textrm{ and }\Omega \subset \mathbb{D}_{R,\rho}.
$$
\par Although a refined study has been led in Section \ref{sec:magdata} to obtain explicit weak solutions to \eqref{YF:CB}, whence to \eqref{YF:divu}, with boundary data $u_0$ and $u_1$ belonging to a larger space as usually (the trace of a generalised Hardy space), and since we want to use standard technics of Shape Optimization, we make the classical assumptions on the boundary data, that is $u_0\in W^{1/2,2}(\Gamma_e)$ and $u_1\in W^{-1/2,2}(\Gamma_e)$.
\par In other words, $\Omega$ can be seen as the solution, if it exists, of the  ``free boundary problem''
\begin{equation}\label{eq1}
\left\{
\begin{array}{ll}
\nabla \cdot \left(\sigma \nabla u \right)=0, & x\in \Omega\\
u=u_0, & x\in \Gamma_e \\
\frac{\partial u}{\partial n}=u_1, & x\in \Gamma_e\\
u=c & x\in \Gamma_p ,
\end{array}\right.
\end{equation}
where $c>0$ is a constant chosen so that $M_0\in \Gamma_p$. Written under this form, this problem looks like a free boundary problem and we could have a temptation to apply Holmgren's uniqueness or a Cauchy-Kowalewski's type theorem. Nevertheless and to our knowledge (see for instance \cite{hormander}), until yet, that is only possible to deduce from these theorems, the existence of a neighborhood of the external boundary $\Gamma_e$ where is defined a local solution of the three first equations of \eqref{eq1}, but the existence of a closed level set $\Gamma_p$ inside the domain delimited by $\Gamma_e$ seems to be a difficult question.
\par Let us say one word on the uniqueness of solutions for this ``free boundary problem'', that is the identifiability of the interior boundary curve $\Gamma_p$ from one pair of Cauchy data on the exterior boundary curve. This question has been already studied, for instance in \cite{Kress1,Kress2} and the proof fits immediately into our frame: $u_0$ and $u_1$ being fixed, if two domains $\Omega_1$ and $\Omega_2$ respectively associated with the interior plasma boundary $\Gamma_p^1$ and $\Gamma_p^2$ satisfy the overdetermined System \eqref{eq1}, and if $\Gamma_p^1$ and $\Gamma_p^2$ have both a Lipschitz regularity, then $\Omega_1=\Omega_2$.
\par Notice that a close family of ``free boundary problems'' have been studied in the past, referring to overdetermined like problems with additional Bernoulli type boundary condition. However, in these problems, the Bernoulli condition applies to the free boundary and permits in general to ensure some kind of regularity of this latter and its uniqueness. The case presented in this section is quite unusual in the context of Shape Optimization since the overdetermined condition applies on the fixed boundary $\Gamma_e$. Let us notice nevertheless that this kind of inverse problem has been strongly studied with different approaches and tools (conformal mapping, integral equations, quasi-reversibility, level-sets and so on) and we mention for instance \cite{Ameur,bourgeois,burger,Haddar,Rundell}. About the literature on free boundary problem with a Bernoulli condition, we refer to \cite{hasl1,hasl2,henrot_shah1,henrot_shah2,henrot_shah3,Ito,LauPri,lindgren}.
\par Furthermore and with respect to the physical experiments, we will make the following assumptions: 
\begin{equation}\label{eqAss}
u_1<0\textrm{ and }u_0<c\textrm{ a.e. on }\Gamma_e,
\end{equation}
coming from the physical frame. Let us notice that, thanks to these additional informations, if we assume moreover that the domain $\Omega$ satisfies the interior sphere property condition, then the application of Hopf's maximum principle yields
$$
\min_{\overline{\Omega}}u=\min_{\Gamma_e} u_0 \textrm{ and } \frac{\partial u}{\partial n}(\textrm{argmin }u)=u_1(\textrm{argmin }u)<0.
$$ 
In the same way, one has also
$$
\max_{\overline{\Omega}}u=\max_{\Gamma_p} u=c, \textrm{ and }\frac{\partial u}{\partial n}_{\mid_{\Gamma_p}}>0.
$$ 
\par For the reasons mentioned before, we decided to see this ``free boundary problem'' as a Shape Optimization one. Indeed, let us introduce, $\Omega$ and $c>0$ being fixed, the solution $u_\Omega$ of the following mixed Dirichlet-Neumann elliptic problem
\begin{equation}\label{eq2}
\left\{
\begin{array}{ll}
\nabla \cdot \left(\sigma \nabla u \right)=0, & x\in \Omega\\
\frac{\partial u}{\partial n}=u_1, & x\in \Gamma_e \\
u=c & x\in \Gamma_p ,
\end{array}\right.
\end{equation}
Let us denote by $D$ the closure of $\mathbb{D}_{R,\rho}$, i.e. the compact set which boundary is $\Gamma_e$. If the capacity of the compact set $D\backslash\Omega$ is strictly positive and $u_1$ belongs to $W^{-1/2,2}(\Gamma_e)$, then this problem has obviously a unique solution, by virtue of Lax-Milgram's Theorem.
\par Now, a natural idea to solve the overdetermined system \eqref{eq2} consists in introducing a least square type shape functional $J$ through the Shape Optimization problem 
\begin{equation}\label{eq3}
\left\{
\begin{array}{l}
\min J(\Omega)=\int_{\Gamma_e}\left(u-u_0\right)^2 ds\\
\Omega \in \mathcal{O}_{\textnormal{ad}}, 
\end{array}\right.
\end{equation}
where $\mathcal{O}_{\textnormal{ad}}$ denotes the set of admissible shapes that has to be clarified. The natural questions arising here are:
\begin{itemize}
\item How to choose $\mathcal{O}_{\textnormal{ad}}$ to ensure the existence of solutions for Problem \eqref{eq3}?
\item In the case where Problem \eqref{eq3} has a solution $\Omega^\star$, can we guarantee that $J(\Omega^\star)=0$?
\end{itemize}
In the following section, we will bring some partial answers to these questions.
\subsection{An existence result}
This section is devoted first to the statement of an existence result for the Shape Optimization problem \eqref{eq3} and secondly to the writing of the first order necessary optimality conditions of this problem. We need to define some convergence notions for the elements of $\mathcal{O}_{\textrm{ad}}$, providing a topology within this class.
\begin{definition}
A sequence of open domains $(\Omega_n)_{n\geq 0}$ is said
\begin{itemize}
\item converging to $\Omega$ for the Hausdorff convergence if
$$
\lim_{n\to +\infty}\textrm{d}_H(D\backslash \Omega_n,D\backslash \Omega)=0,
$$
where $\textrm{d}_H(K_1,K_2)=\max (\rho(K_1,K_2),\rho(K_2,K_1))$, for any $(i,j)\in \{1,2\}^2$, $\rho(K_i,K_j)=\sup_{x\in K_i}d(x,K_j)$, and $\forall x\in D$, $d(x,K_i)=\inf_{y\in K_i}d(x,y)$.
\item converging to $\Omega$ in the sense of characteristic functions if for all $p\in [1,+\infty)$,
$$
\chi_{\Omega_n}\xrightarrow[n\to\infty]{}\chi_{\Omega}\textrm{ in }L^p_{\textrm loc}(\mathbb{R}^2).
$$
\item converging  to $\Omega$ in the sense of compacts if
\begin{enumerate}
\item $\forall K\textnormal{ compact subset of }D, K\subset \Omega \Rightarrow \exists n_0 \in \mathbb{N}^*, \ \forall n\geq n_0, \  K\subset \Omega_n$.
\item $\forall K\textnormal{ compact subset of }D, K\subset D\backslash\overline{\Omega} \Rightarrow \exists n_0 \in \mathbb{N}^*, \ \forall n\geq n_0, \  K\subset D\backslash\overline{\Omega_n}$.
\end{enumerate}
\end{itemize}
\end{definition}
It is in general a hard task to get existence results in Shape Optimization. Not only, many such problems are often ill-posed (see for instance \cite{allaire,HP}), but if the class of admissible domains is too large, a minimizing sequence of domains may converge to a very irregular domain, for which the solution of the partial differential equation may exist in a very weak sense. For these reasons, a partial solution consists in restricting the class of admissible domains, by assuming some kind of regularity. For that purpose, let us define the notion of $\varepsilon$-cone property, introduced in \cite{chenais}.
\begin{definition}
Let $y$ be a point of $\mathbb{R}^2$, $\xi$ a normalized vector and $\varepsilon>0$. We denote by $C(y,\xi,\varepsilon)$, the unpointed cone 
$$
C(y,\xi,\varepsilon)=\{z\in \mathbb{R}^2, \langle z-y,\xi\rangle \geq \cos\varepsilon \Arrowvert z-y\Arrowvert\textrm{ and }0<\Arrowvert z-y\Arrowvert<\varepsilon \}.
$$
We say that an open set $\Omega$ verifies the $\varepsilon$-cone property if
$$
\forall x\in\partial\Omega, \exists \xi_x\in\mathbb{S}^1, \forall y\in\overline{\Omega}\cap B(x,\varepsilon), C(y,\xi_x,\varepsilon)\subset\Omega .
$$
\end{definition}
Let us make precise the class of admissible domains. For a given $\eta >0$, we call $D_\eta$, the annular compact set $D\backslash B(X_0,\eta)$, where $X_0=(0,R)$ is the center of the disk $D$. For $\eta$ and $\varepsilon$, two fixed strictly positive numbers, we define
\[
\mathcal{O}_{\textrm{ad}}^{\varepsilon,\eta}=\{\Omega \text{ open subset of } D_\eta \text{ verifying the } \varepsilon \text{ cone property}, \Gamma_e\subset \partial \Omega, \ \#(D_\eta\backslash\Omega)\leq 1\},
\]
where $\#(D_\eta\backslash\Omega)$ denotes the numbers of connected components of the complementary set of $\Omega$ in $D_\eta$, and we replace the class $\mathcal{O}_{\textrm{ad}}$ by $\mathcal{O}_{\textrm{ad}}^{\varepsilon,\eta}$ in Problem \eqref{eq3}.
One has the following existence result.
\begin{proposition}
Let $\eta$ and $\varepsilon$ be two fixed strictly positive numbers. Then, Problem \eqref{eq3} admits at least one solution.
\end{proposition}
To prove this proposition, we follow the direct method of calculus of variations. Although this technic is standard, we hope that this result, in the particular frame of plasma Physics, is new.  
\begin{proof}
Let $(\Omega_n)_{n\geq 0}$ be a minimizing sequence for Problem \eqref{eq3} and let $\Gamma_{p,n}$ be the part of the boundary contained in $\Omega_n$, that is $\partial\Omega_n\backslash \Gamma_e$. Since for any $n\in\mathbb{N}$, $\Omega_n$ is included in the compact set $D_\eta$, there exists $\Omega^\star$ such that $(\Omega_n)_{n\geq 0}$ converges to $\Omega^\star$ for the Hausdorff metric and in the sense of characteristic functions (see \cite[Chapter 2]{HP}). Let us denote by $\Gamma_p^\star$ its internal boundary. Furthermore, since any $\Omega_n$ satisfies an uniform $\varepsilon$ cone property, the sequence $(\Omega_n)_{n\geq 0}$ converges to $\Omega^\star$ in the sense of compacts and necessarily, $\Omega^\star$ satisfies the uniform $\varepsilon$-cone property \cite[Theorem 2.4.10]{HP}.
\par To prove the existence result, we need to prove at least the lower semi-continuity of the criterion $J$ in the class $\mathcal{O}_{\textrm{ad}}^{\varepsilon,\eta}$. 
\par Let us first prove that the sequence $(u_{\Omega_n})_{n\geq 0}$ is bounded in $W^{1,2}(D_\eta)$. Notice that due to the homogeneous Dirichlet boundary condition on the lateral boundary $\Gamma_{p,n}$, we can extend $u_{\Omega_n}$  by $c$ outside $\Gamma_{p,n}$. So we can consider that the functions are all defined on the box $D_\eta$ (or $D$) and the integrals over $\Gamma_{p,n}$ and over $D_\eta$ (or $D$) will be the same. For that purpose, let us multiply Equation \eqref{eq2} by $u_{\Omega_n}-c$ and then integrate by part. From Green's formula, we get
$$
C_1\int_{\Omega_n}|\nabla u_{\Omega_n}|^2 \leq \int_{\Omega_n}\sigma|\nabla u_{\Omega_n}|^2 =\int_{\Gamma_e}\sigma u_1(u_{\Omega_n}-c),
$$
and the right hand side is uniformly bounded with respect to $n$, due to the fact that the sequence $(J(\Omega_{n}))_{n\geq 0}$ is convergent and hence bounded. Now, it remains to prove that the sequence $(u_{\Omega_n})_{n\geq 0}$ is bounded in $L^2(\Omega_n)$ (in the sense that $\int_{\Omega_n}u_{\Omega_n}(x)^2dx$ is uniformly bounded with respect to $n$). This is actually a consequence of the non positivity of $u_1$. Indeed, the use of Hopf's Theorem yields immediately that $u_{\Omega_n}$ attains its maximum value at $x_m\in\Gamma_e\cup \Gamma_{p,n}$ and that $\frac{\partial u_{\Omega_n}}{\partial n}(x_m)>0$ and hence, $\max_{\overline{\Omega_n}} u_{\Omega_n}={u_{\Omega_n}}_{\mid_{\Gamma_{p,n}}}=c$. Notice that the regularity condition on $\Omega_n$ required to apply Hopf's lemma is satisfied by choosing every $\Omega_n$ having a regular enough boundary, which is clearly not restrictive. Let us introduce $u_{\eta/2}$, the solution of \eqref{eq2} set in the interior of $D_{\eta/2}$. A comparison between $u_{\Omega_n}$ and $u_{\eta/2}$, using Hopf's maximum principle ensures that for any $n\in\mathbb{N}$, $u_{\eta/2}\leq u_{\Omega_n}\leq c$ a.e. in $D$. Thus, $\beta=\max(\Arrowvert u_{\eta/2}\Arrowvert_{L^2(D)},\pi \rho^2c)$ is an uniform bound of the $L^2(D)$ norm of $u_{\Omega_n}$, which proves the $W^{1,2}(D)$ boundedness of $u_{\Omega_n}$, and shows at the same time that $(u_{\Omega_n})_{n\geq 0}$ is bounded in $W^{1,2}(D_\eta)$.
\par As a result, using Banach Alaoglu's and Rellich-Kondrachov's Theorems, there exists a function $u^\star\in W^{1,2}(D_\eta)$ such that
$$
u_{\Omega_n}\stackrel{W^{1,2}(D_{\eta/2})}{\rightharpoonup}u^\star\textrm{ and }u_{\Omega_n}\stackrel{L^{2}(D_{\eta/2})}{\rightarrow}u^\star\textrm{ as }n\to+\infty .
$$
It remains to prove that $u^\star$ is the solution of Equation \eqref{eq2} on $\Omega^\star$. Let us notice that, once this assertion will be proved, the continuity of the shape functional $J$ will immediately follow, because of the above convergence result. Let us write the variational formulation of \eqref{eq2}. For any function $\varphi$ satisfying $\varphi\in W^{1,2}(D_{\eta/2})$, $\varphi=0$ on $\Gamma_e\cup\Gamma_{p}^\star$, and for all $n \in\mathbb{N}$, the function $u_{\Omega_n}$ verifies
\begin{equation}\label{FVun}
\int_{D}\sigma \nabla u_{\Omega_n}\cdot \nabla \varphi =\int_{\Gamma_e}\sigma u_1\varphi .
\end{equation}
It suffices to use the weak $H^1$-convergence  of $u_{\Omega_n}$ to $u^\star$ to show that $u^\star$ satisfies also \eqref{FVun}. To conclude, it remains to prove that $u^\star=c$ on $\Gamma_p^\star$. This is actually a consequence of the convergence of $(\Omega_n)_{n\geq 0}$ into $\Omega^\star$ and the fact that $\Omega^\star$ has a Lipschitz boundary. We refer to Theorem 2.4.10 and Theorem 3.4.7 in \cite{HP}. Finally, let us notice that the inclusion $\Gamma_e\subset \partial \Omega^\star$ is verified. Indeed, it follows from the stability of the inclusion for the Hausdorff convergence. To be convinced, one way consist in considering that $\Gamma_e$ is the boundary of an open set $\Sigma$ that does not contain $D$, and to replace the constraint $\Gamma_e\subset \partial \Omega$ by $\Sigma \subset \widetilde{\Omega}$, where $\widetilde\Omega=\Omega\cup\Gamma_e\cup\Sigma$. Hence, the stability of the inclusion of open sets for the Hausdorff convergence applies and the conclusion follows. 
\end{proof}
Notice that the existence result do not guarantee that the solution $u_{\Omega^\star}$ of \eqref{eq2} set in the optimal domain $\Omega^\star$ satisfies the Cauchy conditions on $\Gamma_e$, that is the nonhomogeneous Dirichlet and Neumann boundary conditions on $\Gamma_e$.
\subsection{First order optimality conditions}
Let us now write the first order optimality conditions for Problem \eqref{eq3}. For that purpose, we will first compute the shape derivative of the criterion $J$. The major difficulty when dealing with sets of shapes is that they do not have a vector space structure. In order to be able to define shape derivatives and study the sensitivity of shape functionals, we need to construct such a structure for the shape spaces. In the literature, this is done by considering perturbations of an initial domain; see \cite{allaire,HP,Mu-Si,So-Zo}. 
\par Practically speaking, it is preferable to enlarge the class of admissible domains to write the first order optimality conditions, even if there is a risk of loosing the existence result. Indeed, it would be else strongly difficult for instance to take into account the fact that any admissible domain must satisfy an $\varepsilon$-cone property. For this reason, let us define
\begin{equation}\label{Oad}
\mathcal{O}_{\textrm{ad}}=\{\Omega\textrm{ open subset of $D$}, \#(D_\eta\backslash\Omega)\leq 1\}.
\end{equation}
Let $\Omega\in\mathcal{O}_{\textrm{ad}}$ given. Let us consider a regular vector field $V:\mathbb{R}^2\to \mathbb{R}^2$ with compact support, which does not meet the fixed boundary $\Gamma_e$. For small $t>0$, we define $\Omega_t=(I+tV)\Omega$, the image of $\Omega$ by a perturbation of Identity and $f(t):=J(\Omega_t)$. We recall that the shape derivative of $J$ at $\Omega$ with respect to $V$ is 
$$
f'(0)=\lim_{t\searrow 0}\frac{J(\Omega_t)-J(\Omega)}{t}.
$$
We will denote it by $dJ(\Omega ;V)$. Let us now compute this shape derivative. The classical formulae yield
$$
dJ(\Omega ;V)=2\int_{\Gamma_e}(u_\Omega-u_0)\dot u ds,
$$
where $\dot u$ denotes the shape derivative of $u$, and is solution of the system
\begin{equation}\label{eq4}
\left\{
\begin{array}{ll}
\nabla \cdot \left(\sigma \nabla \dot u\right)=0, & x\in \Omega\\
\frac{\partial \dot u}{\partial n}=0, & x\in \Gamma_e\\
\dot u=-\frac{\partial u_\Omega}{\partial n}(V\cdot n) & x\in \Gamma_p.
\end{array}\right.
\end{equation}
The proof of such an assertion may be found for instance in \cite[Section 3.1]{IKP}.
It is more convenient to work with another expression of the shape derivative and to write it as a distribution with support $\Gamma_p$. For that purpose, let us introduce the adjoint state $p$ defined formally as the solution of the system
\begin{equation}\label{adjointp}
\left\{
\begin{array}{ll}
\nabla \cdot \left(\sigma \nabla p\right)=0, & x\in \Omega\\
\sigma \frac{\partial p}{\partial n}=2(u_\Omega-u_0), & x\in \Gamma_e\\
p=0 & x\in \Gamma_p.
\end{array}\right.
\end{equation}
\begin{proposition}\label{PropAdjoint}
Assume that $\partial\Omega$ is $C^2$ (in other words, $\Gamma_p$ is $C^2$). Then, the criterion $J$ is differentiable with respect to the shape at $\Omega$ and one has for all admissible perturbation $V\in W^{1,\infty}(\mathbb{R}^2,\mathbb{R}^2)$,
$$
dJ(\Omega;V)=\int_{\Gamma_p}\sigma \frac{\partial p}{\partial n}\frac{\partial u_\Omega}{\partial n}(V\cdot n)ds.
$$
As a consequence, if Problem \eqref{eq3} has a solution $\Omega^\star$ having a internal boundary $\Gamma_{i}^\star$ which is a $C^2$ closed curve, then $\Omega^\star$ solves the overdetermined problem \eqref{eq1}.
\end{proposition}
\begin{proof} The shape differentiability is standard (see \cite{DZ,HP}) under the assumption that the boundary of $\Omega$ is regular enough. Let us now prove the expression of the shape derivative. First multiply equation (\ref{eq4}) by $p$ and integrate on $\Omega$. The use of the Green formula yields
$$
-\int_{\Omega}\sigma \nabla p\cdot \nabla \dot u =-\int_{\Gamma_e}\sigma \frac{\partial \dot u}{\partial n}p =0.
$$
Multiplying equation (\ref{adjointp}) by $\dot u$ and integrating by parts yields
$$
-\int_{\Omega}\sigma \nabla p\cdot \nabla \dot u =-\int_{\Gamma_e\cup \Gamma_p}\sigma \frac{\partial p}{\partial n}\dot u .
$$
The conclusion follows easily. Now, let us investigate the first order optimality conditions. Since $\Gamma_{i,\star}$ is Lipschitz and surrounds a domain whose interior is nonempty, we deduce that if $V$ belongs to the cone of admissible perturbations, then so does $-V$. Hence,
$$
\frac{\partial p}{\partial n}\frac{\partial u_\Omega}{\partial n}=0\textrm{ on }\Gamma_p^\star ,
$$
and it follows that there exists a subset of $\Gamma_p^\star$ of nonzero surface measure on which $u_{\Omega^\star}=c$ and $\frac{\partial p}{\partial n}\frac{\partial u_{\Omega^\star}}{\partial n}=0$. Furthermore, by virtue of Holmgren's theorem, necessarily $\frac{\partial u_{\Omega^\star}}{\partial n}\neq 0$ on such a subset. Hence $\frac{\partial p}{\partial n}=0$ on this subset, and since one has also $p=0$ on $\Gamma_p^\star$, a second application of Holmgren's theorem yields $p\equiv 0$. In other words, $u_{\Omega^\star}=u_0$, which is the desired result.
\end{proof}
\subsection{Numerical computations}
The method presented here is a gradient method, based on the use of the shape derivative $dJ(\Omega;V)$ computed in Proposition \ref{PropAdjoint}. It consists basically in finding a deformation of a given domain $\Omega$ guaranteeing the decrease of the criterion $J(\Omega)$. To sum up, the algorithm writes:
\begin{itemize}
\item Initial guess $\Omega_0\in\mathcal{O}_{\textrm{ad}}$ is given;
\item Iteration $k$: $\Omega_{k+1} = (Id + t_kV_k)\Omega_k$ where $V_k$ is a vector field shrewdly chosen and $t_k$ the time step;
\end{itemize}
Practically speaking,  we are looking for a vector field $V_k$ such that
$$
dJ(\Omega_k, V_k) = \int_{\Gamma_p}\nabla J(\Omega_k)\cdot V_k ds< 0,
$$
where
$$
\nabla J(\Omega_k) = -\sigma \frac{\partial u_{\Omega_k}}{\partial n}\frac{\partial p_k}{\partial n}n,
$$
and $p_k$ is the adjoint state, solution of Problem \eqref{adjointp} set on the domain $\Omega_k$.
One possible choice for $V_k$ is 
\begin{equation}\label{Vk}
\int_{\Gamma_p}\nabla J(\Omega_k)\cdot V_k = -\Arrowvert V_k \Arrowvert_{H^1(\Omega_k)}^2.
\end{equation}
For a precise description of this method or others like the level set method and examples of applications in the domain of Shape Optimization, we refer for instance to \cite{allaire,burger,dogan,deGournay,protas}.
To ensure \eqref{Vk}, let us write $V_k$ as the solution of the problem written under variational form
\begin{equation}\label{FVeq1}
\begin{array}{l}
\text{Find } V_k \in \left(H^1(\Omega_k)\right)^2 \text{ such that }V_k=0\text{ on }\Gamma_e\\
\text{and }\int_{\Omega_k}(\nabla V_k : \nabla \varphi +V_k\cdot \varphi )dx = -\int_{\Gamma_p}\nabla J(\Omega_k)\cdot\varphi ds,
\end{array}
\end{equation}
where the test functions $\varphi$ describe the space of functions $H^1(\Omega_k)^2$ verifying $\varphi = 0$ on $\Gamma_e$, and the symbol ``:'' denotes the tensorial notation used to represent the \textit{doubly contracted tensorial product}, i.e. for $A$ and $B$, two smooth vector fields of $\mathbb{R}^2$,
$$
A:B=\sum_{i,j=1}^2A_{ij}B_{ij}=\textrm{tr}(A\otimes B^\top).
$$
Notice that \eqref{FVeq1} is equivalent to the partial differential equation 
\begin{equation}\label{diffusion}
\left\{
\begin{array}{ll}
-\Delta V_k +V_k = 0 & x \in \Omega_k\\
V_k = 0, & x\in \Gamma_e\\
\frac{\partial V_k}{\partial n} = -\nabla J(\Omega_k), & x \in \Gamma_p^k.\\
\end{array}\right.
\end{equation}\
Written under this form, it is easy to see that this partial differential equation is decoupled.
\par Let us summarize the gradient algorithm:
\begin{enumerate}
\item \textit{Initialization: }Choice of an initial domain $\Omega_0$ and of a real number $\varepsilon>0$;
\item \textit{Iteration $k$: }
	\begin{enumerate}
	\item computation of the solutions $u_{k}$ and $p_{k}$ of Problems (\ref{eq2}) and (\ref{adjointp});
	\item evaluation of $\nabla J(\Omega_k)$ on $\Gamma_p^k$;
	\item computation of the solution $V_k$ of Problem \eqref{diffusion};
	\item determination of $\Omega_{k+1} = (Id + t_k V_k)\Omega_k$;
	\end{enumerate}
\item \textit{stopping criterion: }The algorithm stops when $|J(\Omega_{k+1})-J(\Omega_k)|<\varepsilon$
\end{enumerate} \ \\
An alternative approach, inverting the role of Dirichlet and Neumann conditions in \eqref{eq2}, has also been tested. More precisely, the criterion to minimize has been changed to
$$
J(\Omega)=\int_{\Gamma_e}\left(\frac{\partial u_\Omega}{\partial n}-u_1\right)^2ds,
$$
where $u_\Omega$ solves
\begin{equation}\label{eq5}
\left\{
\begin{array}{ll}
\nabla \cdot \left(\sigma \nabla u \right)=0, & x\in \Omega\\
u=u_0, & x\in \Gamma_e \\
u=c & x\in \Gamma_p .
\end{array}\right.
\end{equation}
We observed that both approach gave similar results. Before presenting our numerical results on the tokamak \textsl{Tore Supra}, let us begin with some numerical tests, in order to verify the efficiency of our method.
\par A first test case has been done, taking $\sigma=1$, and using the fact that the solution of the Laplacian operator is know on an annular domain. Indeed, let us assume that $\Omega=D_2\backslash D_1$ where $D_1$ (respectively $D_2$) denotes the disk centered at the origin of radius $R_1>0$ (respectively $R_2>0$). Then, the solution of   
\begin{equation}\label{test1}
\left\{
\begin{array}{ll}
\Delta u = 0, & x\in \Omega\\
u = 0, & x \in \partial D_1\\
u = c_0 & x \in \partial D_2
\end{array}\right.
\end{equation}\ \\
is given by
$$
u=c_0\frac{\ln \sqrt{x^2+y^2}-\ln R_1}{\ln R_2-\ln R_1}.
$$
Note that the normal derivative of $u$ on $\partial D_2$ is clearly
$$
\frac{\partial u}{\partial n}=\frac{c_0}{R_2(\ln R_2-\ln R_1)}.
$$
A first test can then be made, looking for the internal radius $R_1$, assuming that the normal derivative of $u$ on $\partial D_2$ is constant equal to $c_1>0$ and the expected solution is obviously
$$
R_1=R_2\exp\left(-\frac{c_0}{R_2 c_1}\right).
$$
The results for this case are plotted on Figure \ref{fig1}.
\begin{figure}[htp]
\center
\includegraphics[scale=0.3]{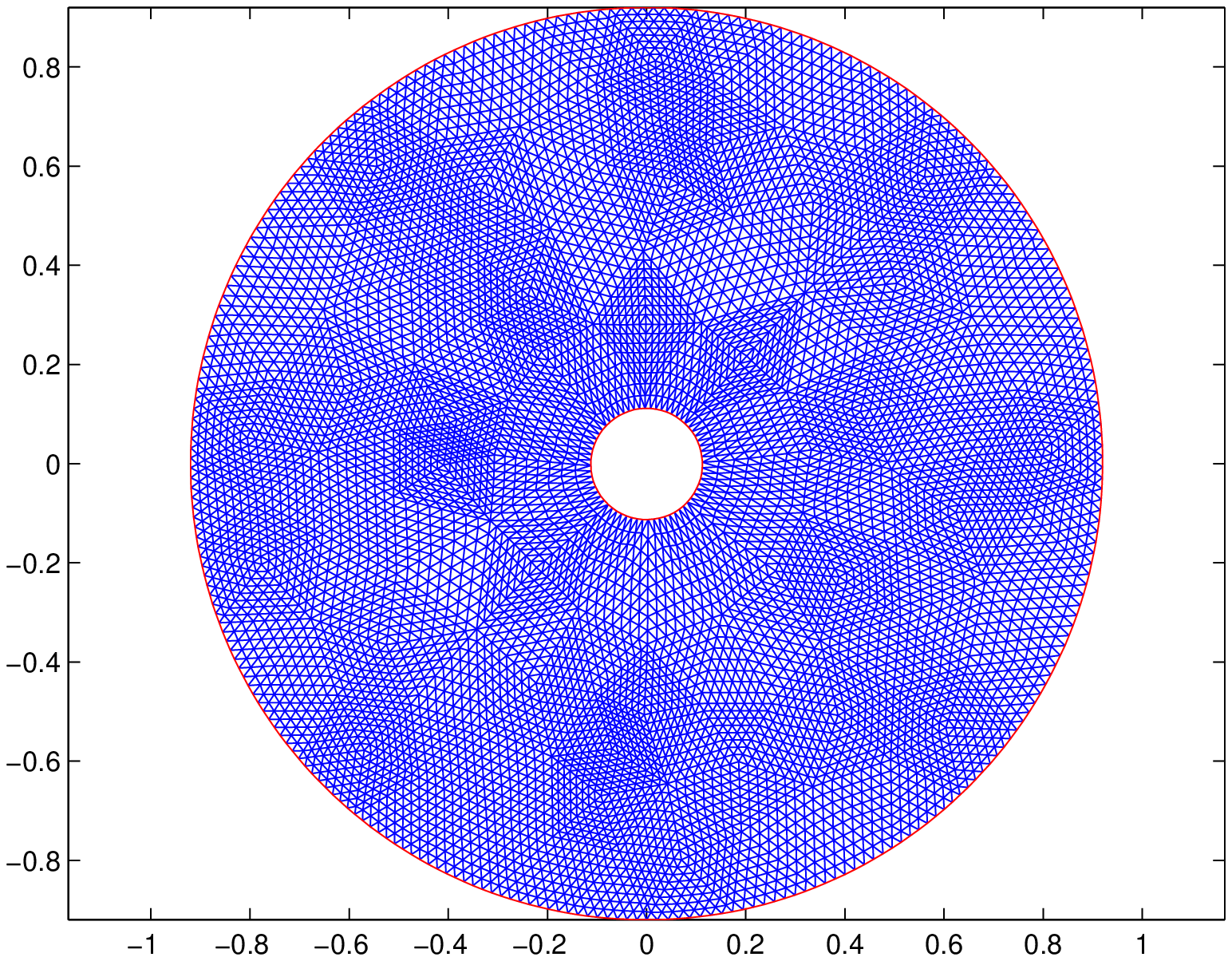}
\includegraphics[scale=0.3]{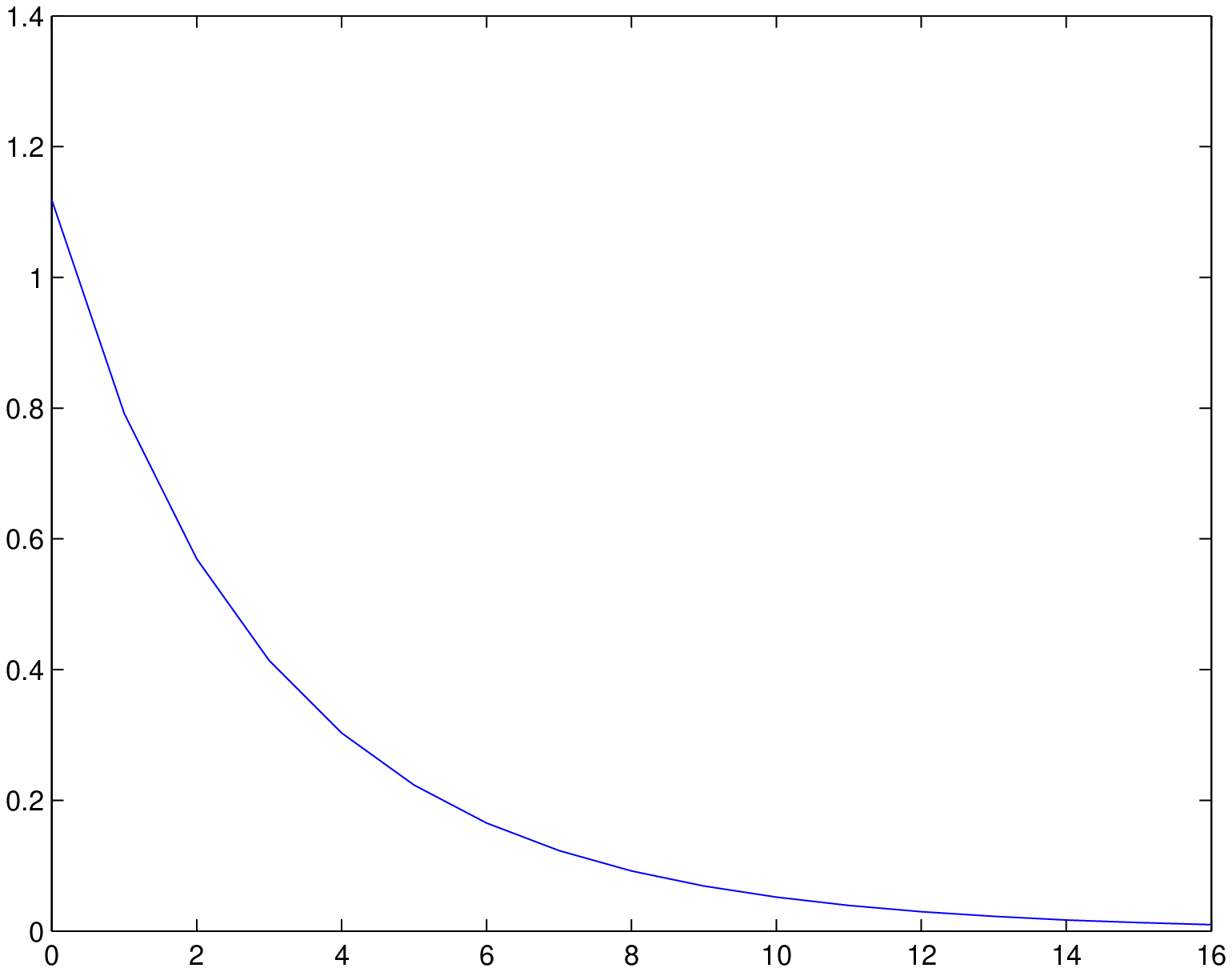}
\includegraphics[scale=0.3]{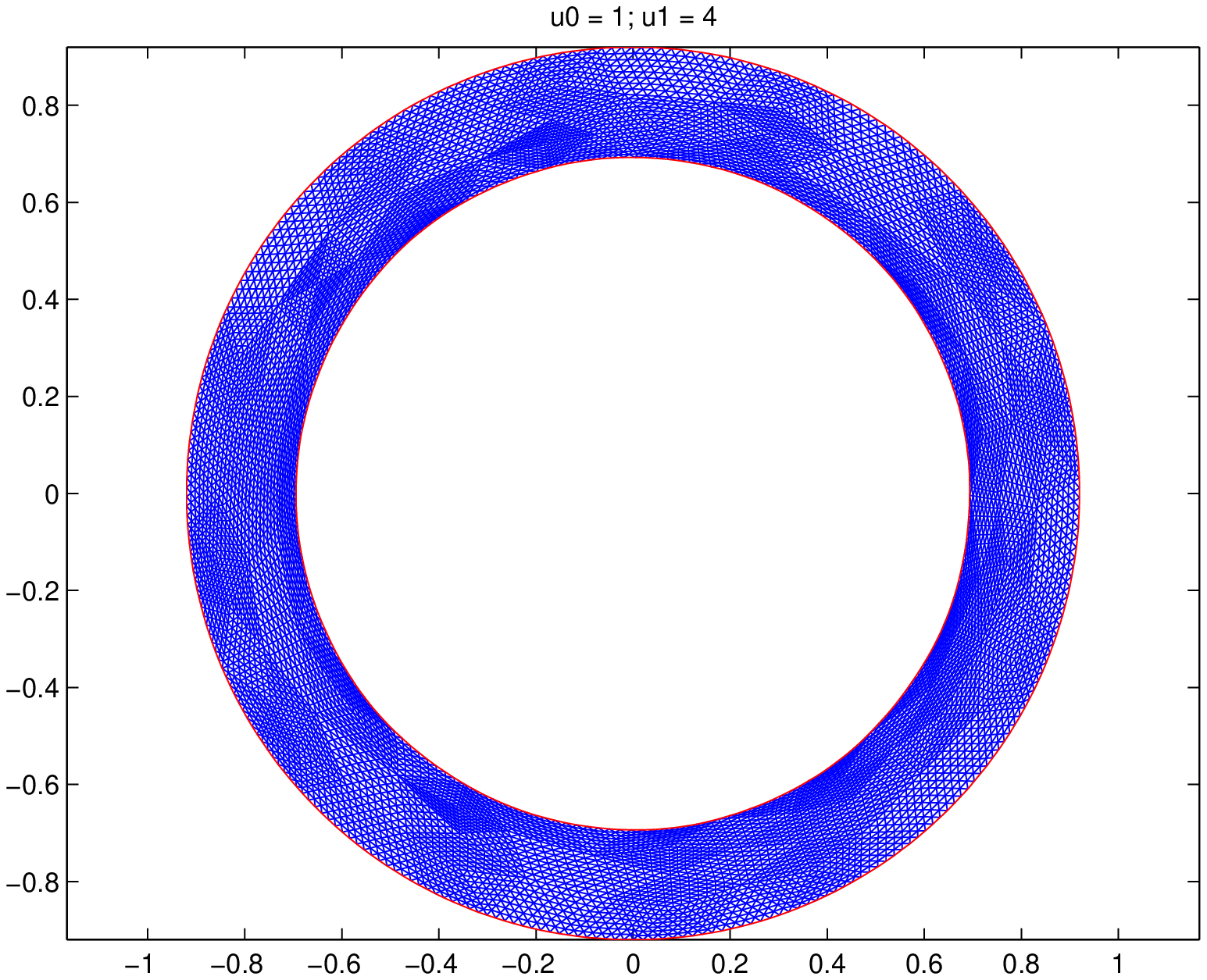}
\includegraphics[scale=0.3]{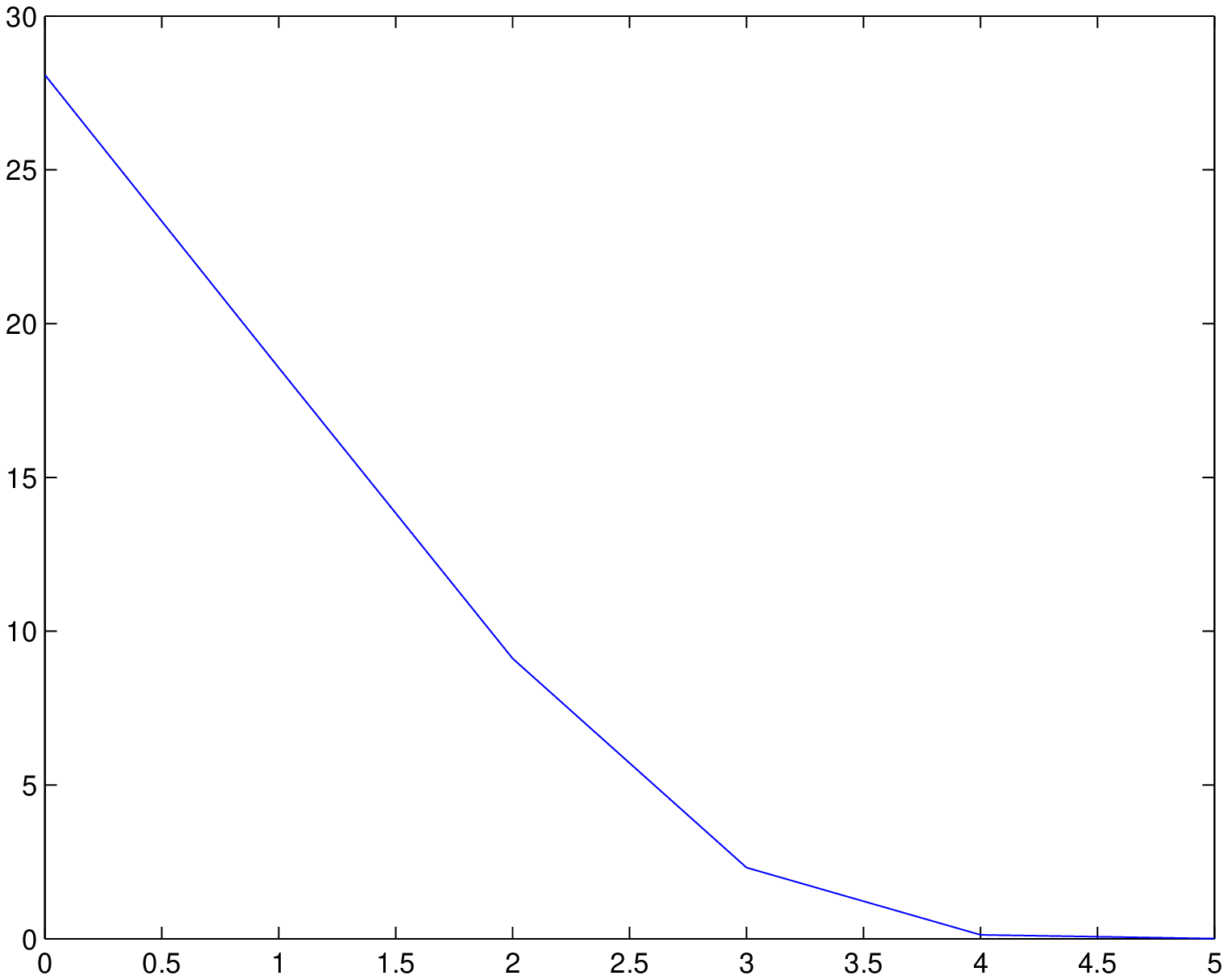}
\caption{Up: optimum computed (left) and criterion (right) for the test case $c_0 = 2$ and $c_1 = 1$. Hence, the expected value of $R_1$ is $0.10$ which is the case. Down: optimum computed (left) and criterion (right) for the test case $c_0 = 1$ and $c_4 = 1$. Hence, the expected value of $R_1$ is $0.70$ which is the case}\label{fig1}
\end{figure}\ \\
\par A second test case has been built, by choosing an arbitrary domain $\Omega_c$ and a function $u_0$, computing the solution $u$ of \eqref{eq2}, and choosing for $u_1$, the quantity $\frac{\partial u}{\partial n}$. Running the program with these definitions of $u_0$ and $u_1$ permits to compare the computed domain with the theoretical domain $\Omega_c$. The results for this case are plotted on Figure \ref{fig2}.
\begin{figure}[htp]
\center
\includegraphics[scale=0.3]{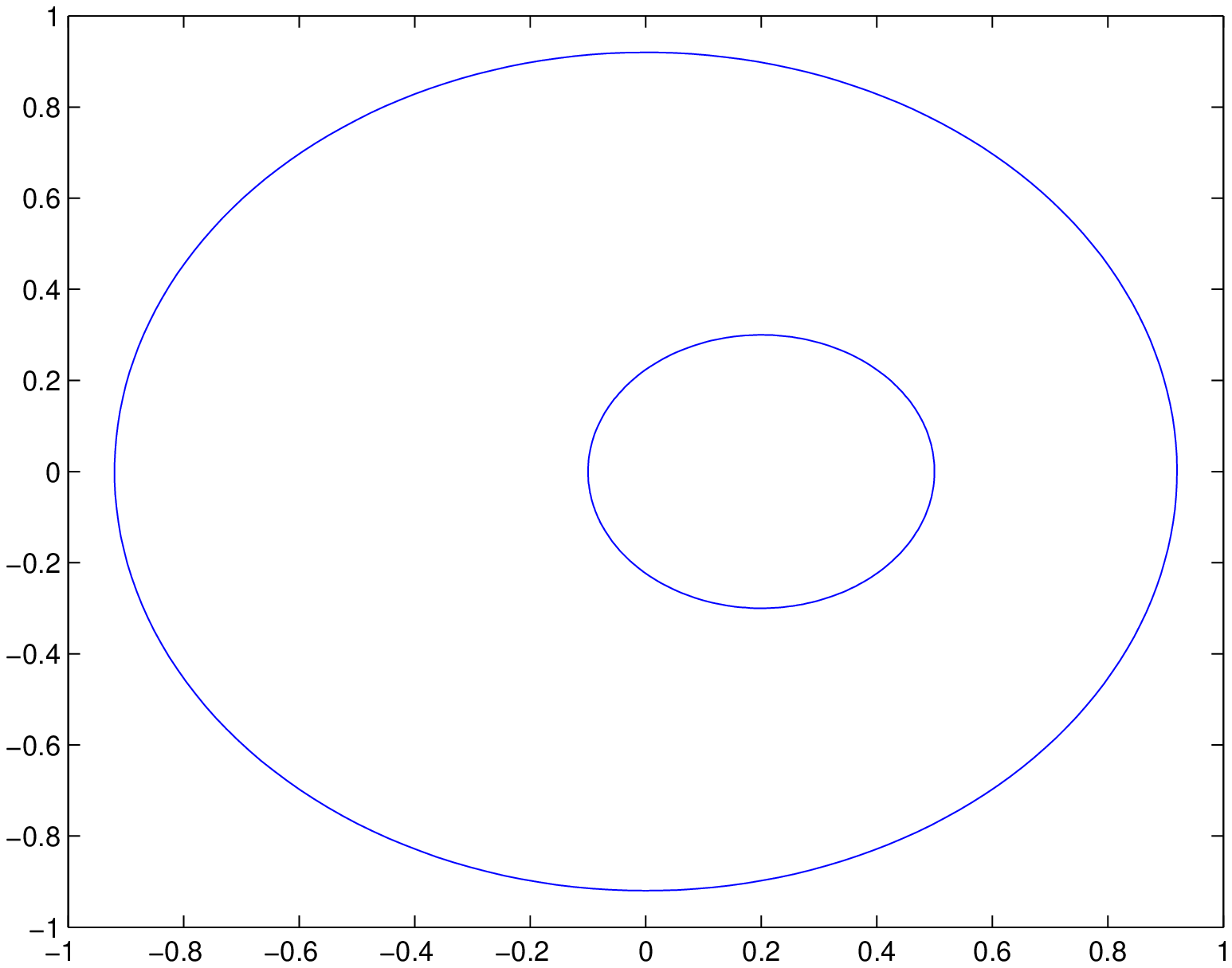}
\includegraphics[scale=0.3]{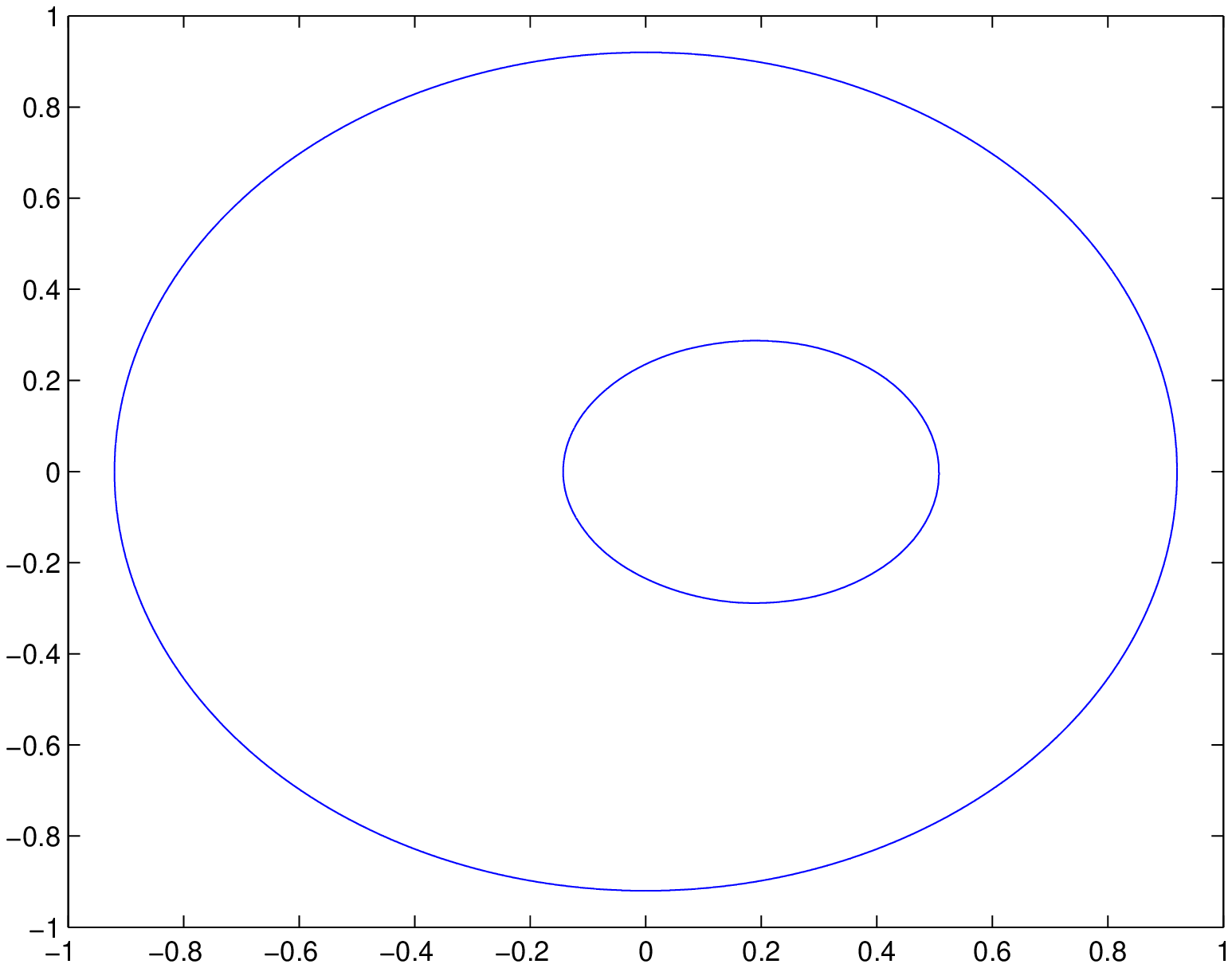}
\includegraphics[scale=0.3]{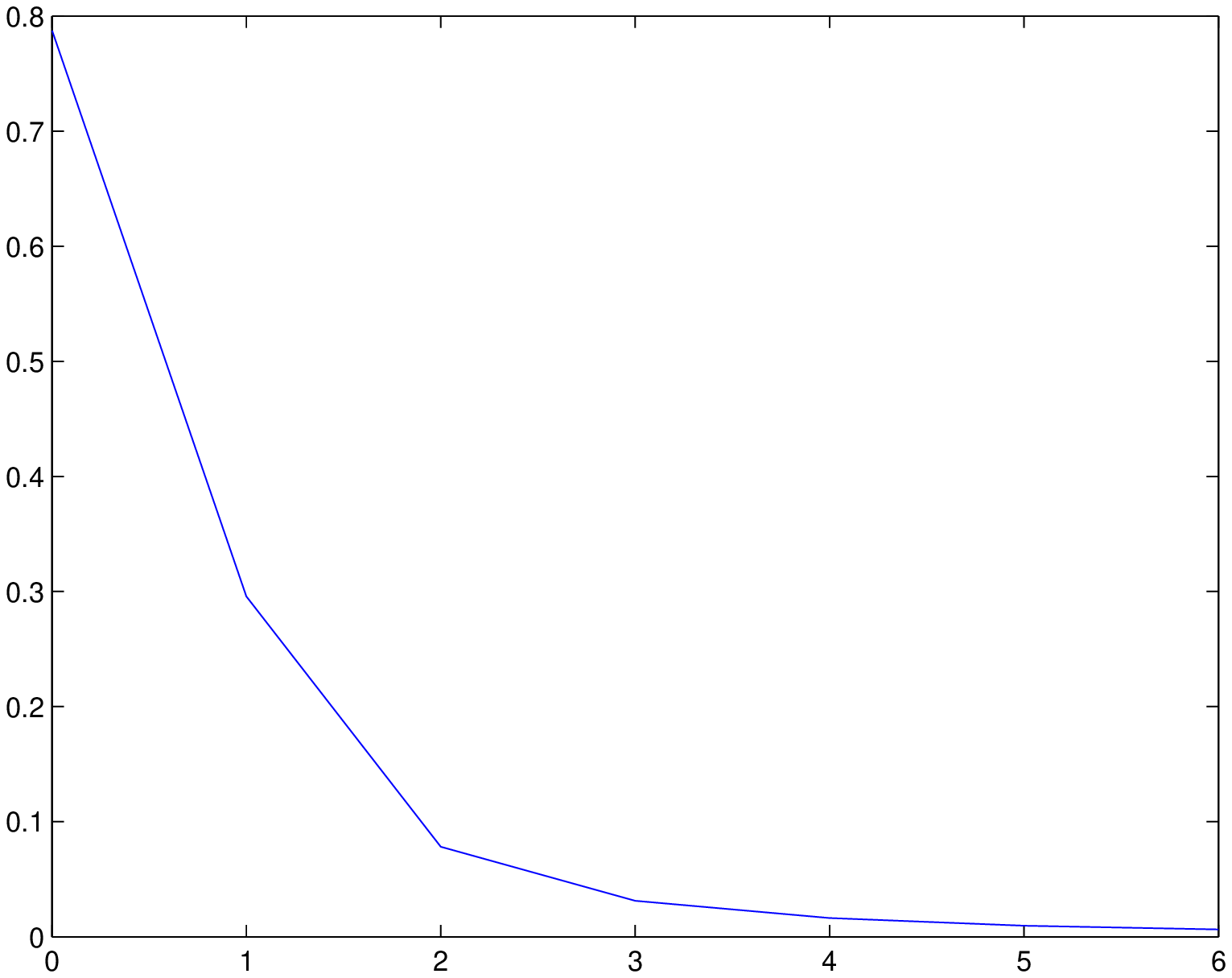}
\caption{Domain $\Omega_c$ expected (left), Domain $\Omega$ computed (right) and criterion (down) in the case $u_0 = -0.2x^2 - 0.5y^2$}\label{fig2}
\end{figure}
\par Let us now present the simulations of the plasma shape in the tokamak {\sl Tore Supra}. 
Measures of $u_0$ and $u_1$ on the external boundary of the tokamak, made by Physicists from CEA/IRFM, are the starting point of our algorithm. We interpolated these data thanks to cubic splines to have a complete definition of these functions  (see Figure \ref{fig3}). All the simulations presented here have been realized using the software Matlab.
\begin{figure}[htp]
\center
\includegraphics[scale=0.3]{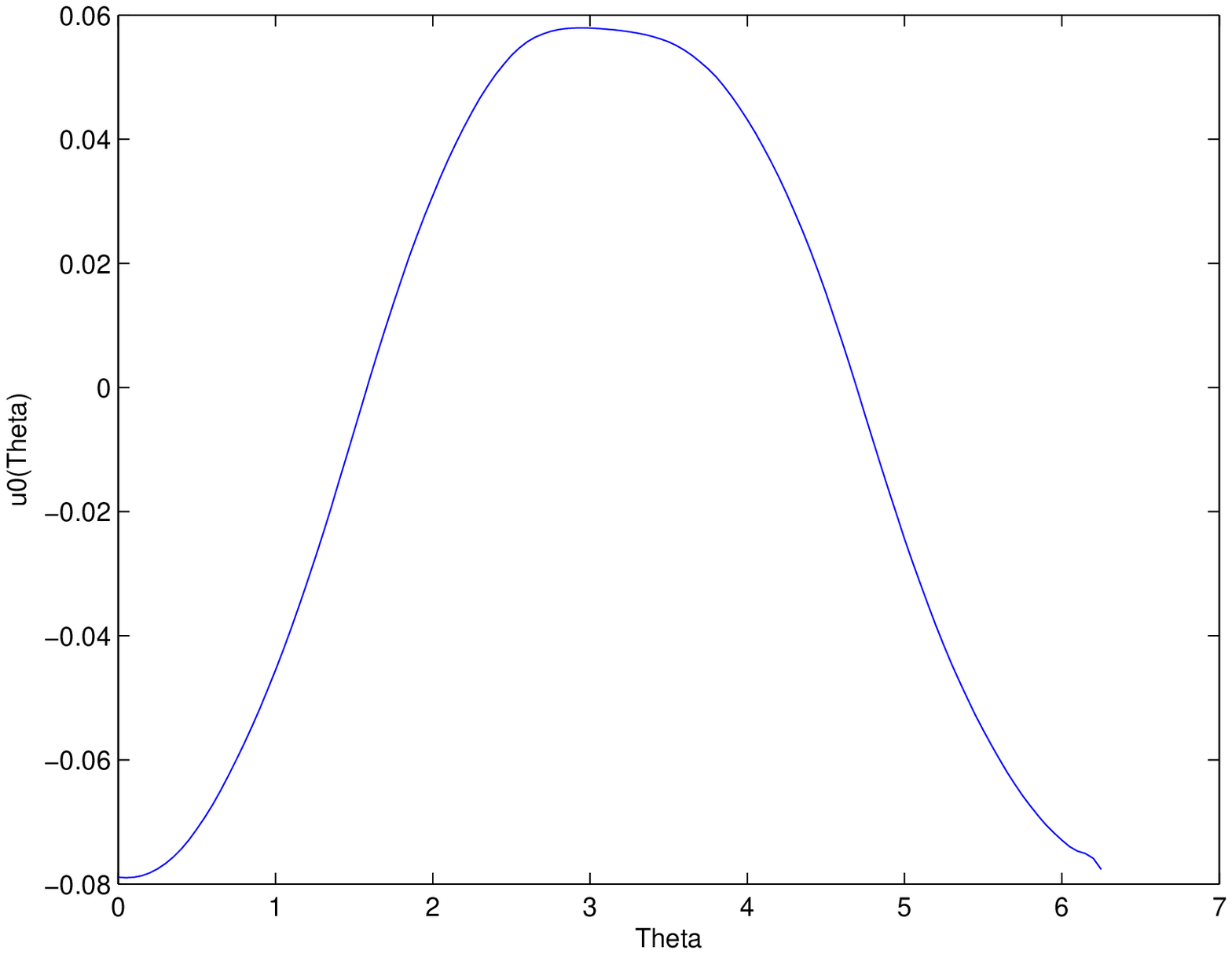}
\includegraphics[scale=0.3]{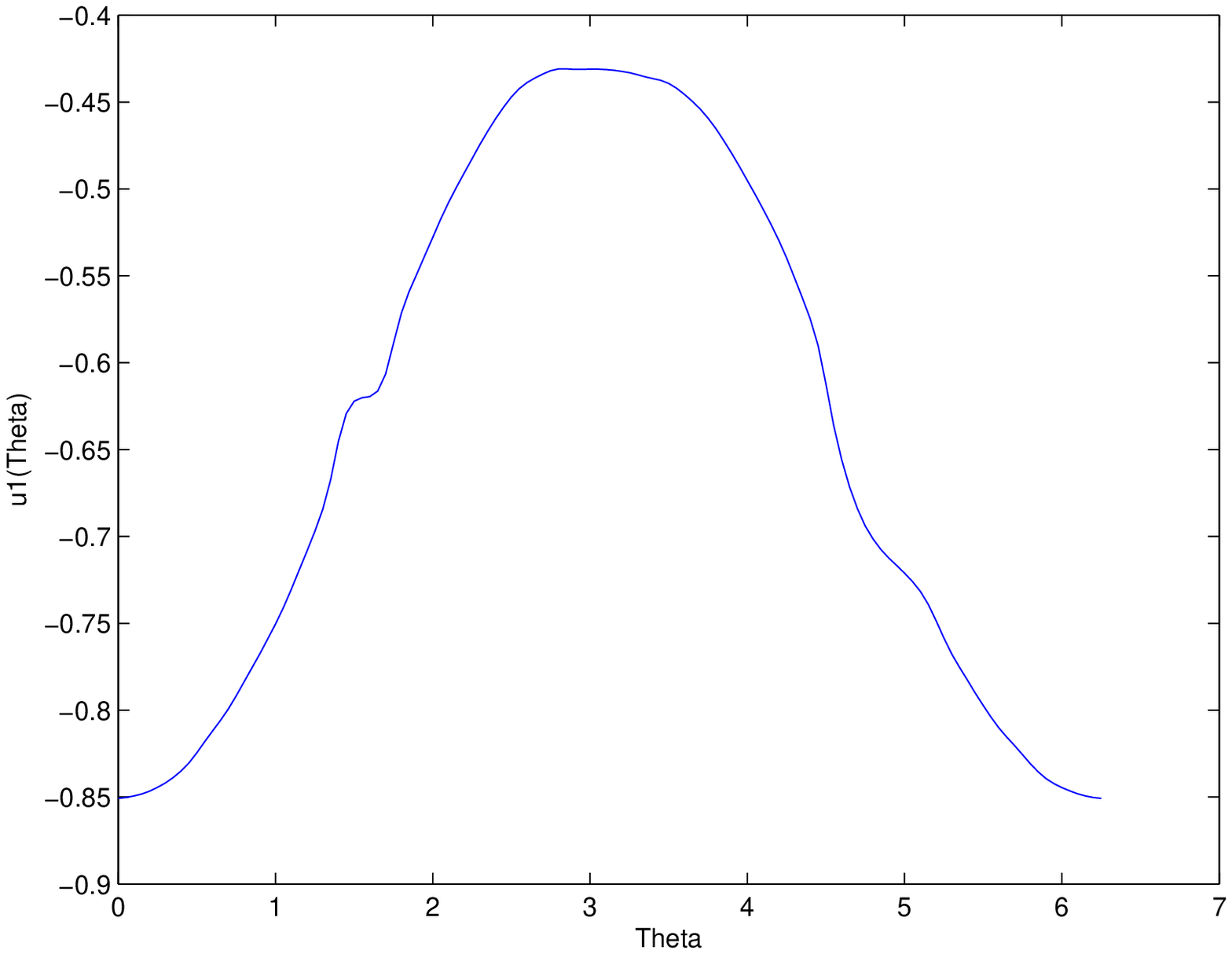}
\caption{$\theta \mapsto u_0(\theta)$ (left) and $\theta \mapsto u_1(\theta) (right)$}\label{fig3}
\end{figure}
With these functions, the algorithm provided the following results plotted on Figure \ref{fig4}.
\begin{figure}[htp]
\center
\begin{tabular}{cc}
\includegraphics[scale=0.38]{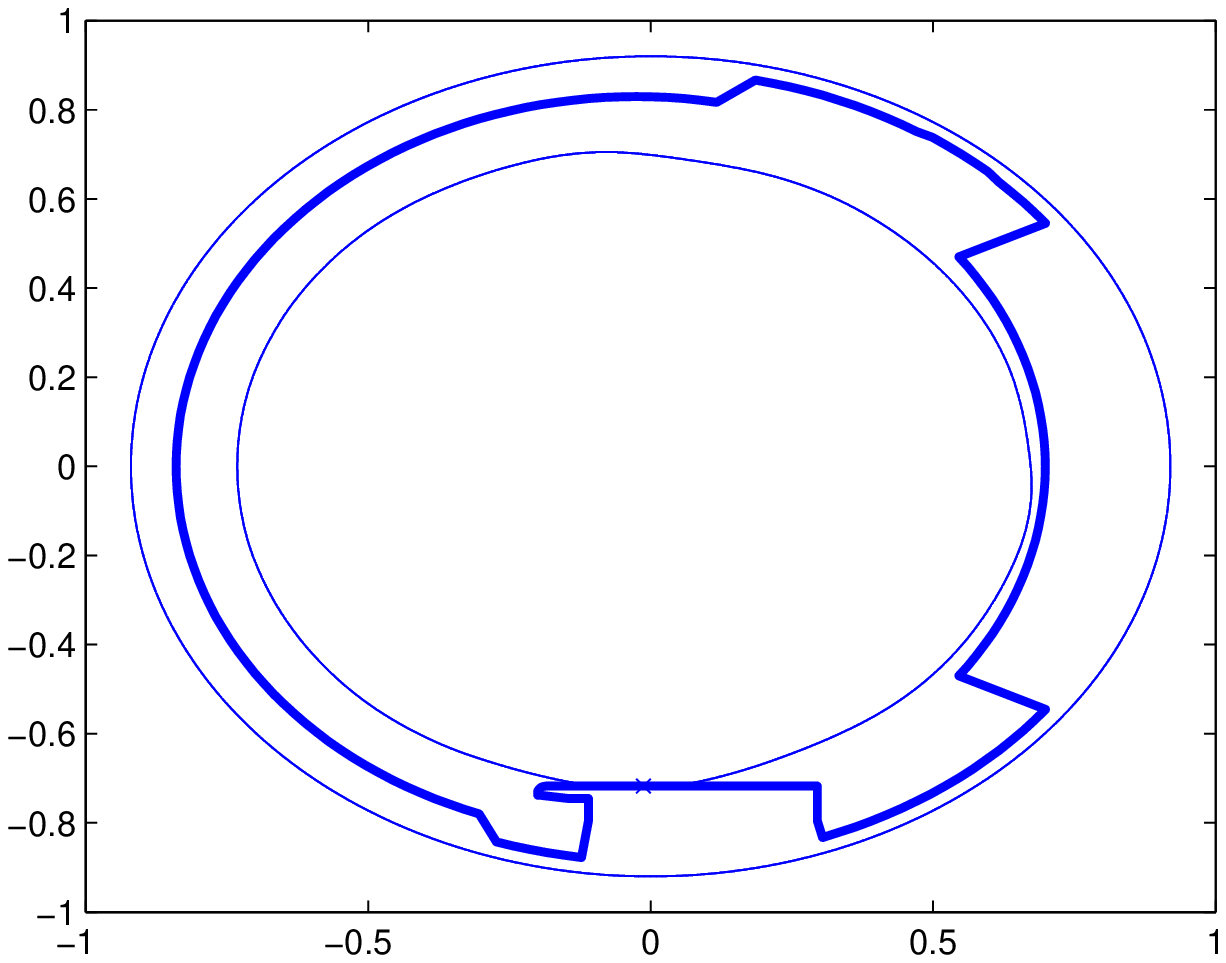} 
&
\includegraphics[scale=0.3]{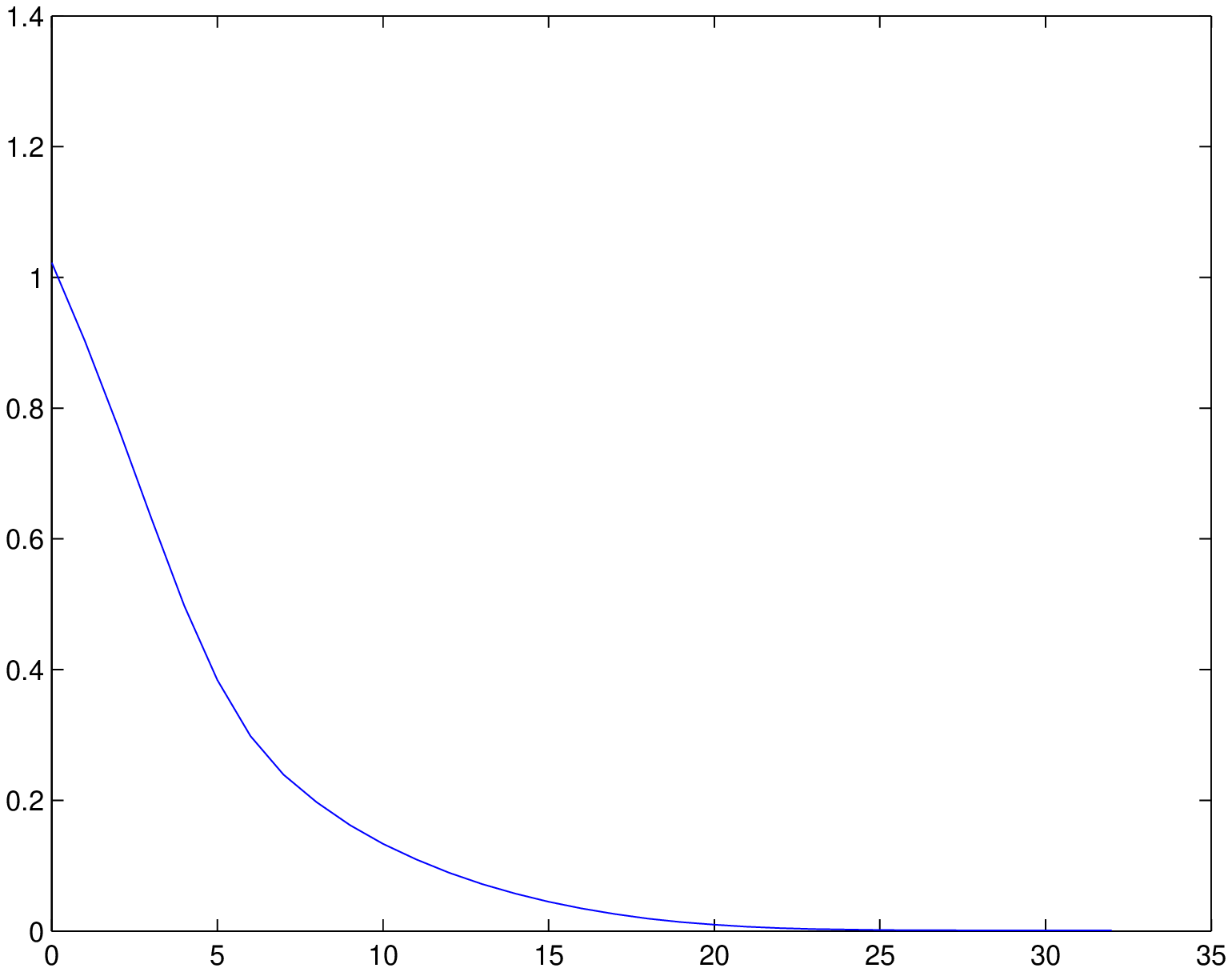}
\end{tabular}
\caption{Closed curves from outside to inside: external boundary, toroidal belt limiter, computed shape of the plasma (left) and criterion (right)}\label{fig4}
\end{figure}
Notice that we did not impose in our algorithm for the plasma to stay inside the limiter, but we fortunately observed it. The point marked by a cross is a control point and the free boundary has to contain this point (it is on the limiter). These conditions being fulfilled, validate our method.
\par A linked interesting inverse problem consists in finding the value of the constant $c$ introduced in Problem \eqref{eq1} from our knowledge of the point $M_0$ belonging to the shape of the plasma. Our algorithm can easily perform the value of such a constant. Indeed, solving Problem \eqref{eq1} with an arbitrary constant $c_1\gg c$ yields a solution $u$ and the associated optimal domain is $\Omega_1=\{u_0\leq u\leq c_1\}$. In fact, the true solution is provided by $\Omega^\star=\{u_0\leq u \leq u(M_0)\}$, where $M_0$ is a point of the limiter where $u$ reaches its maximal value (that is $c$ !). It can be seen as a continuation method up to a flux level set touching the limiter.  Indeed, it is very easy to show that, as a consequence of the maximum principle, $\Omega^\star$ satisfies \eqref{eq0}. On Figure \ref{fig5} is an example of the solution obtained by running the program with $c_1 = 0.2$, the expected constant being closed to $0.153$.
\begin{figure}[htp]
\center
\includegraphics[scale=0.3]{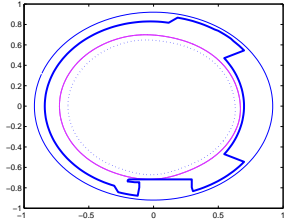}
\caption{Closed curves from outside to inside: external boundary, toroidal belt limiter (bold), $\Omega_{c_1}$ and $\Omega^\star$ (dashed). The constant $c$ evaluated is $0.1569$ whereas the constant $c$ measured by Physicists is $0.1531$}\label{fig5}
\end{figure} \ \\

\textbf{Acknowledgments.} The authors are grateful to Fran\c{c}ois Saint Laurent, engineer of research at CEA/IRFM of Cadarache in the team GPAS, for making available the use of internal data and for all advices and comments. The authors would like to thank Laurent Baratchart, Antoine Henrot and Juliette Leblond for their very constructive suggestions and comments.

\medskip
\end{document}